\documentclass[12pt,reqno]{amsart}
\usepackage[english]{babel}
\usepackage{mathtools,amssymb,amsthm}
\usepackage{latexsym,bm,xcolor}
\usepackage{hyperref}
\usepackage[utf8]{inputenc}
\mathtoolsset{showonlyrefs}
\newcommand{\kahler}{K\"ahler }

\renewcommand{\Im}{\operatorname{Im}}
\renewcommand{\Re}{\operatorname{Re}}
\renewcommand{\epsilon}{\varepsilon}
\let\temp\phi
\let\phi\varphi
\let\varphi\temp

\def\Xint#1{\mathchoice
{\XXint\displaystyle\textstyle{#1}}%
{\XXint\textstyle\scriptstyle{#1}}%
{\XXint\scriptstyle\scriptscriptstyle{#1}}%
{\XXint\scriptscriptstyle\scriptscriptstyle{#1}}%
\!\int}

\def\XXint#1#2#3{{\setbox0=\hbox{$#1{#2#3}{\int}$ }
\vcenter{\hbox{$#2#3$ }}\kern-.6\wd0}}

\def\dashint{\Xint-}

\newcommand{\C}{\mathbb{C}}

\newcommand{\R}{\mathbb{R}}

\newcommand{\Z}{\mathcal{Z}}

\newcommand{\bcal}{B}

\newcommand{\dcal}{\mathcal{D}}

\newcommand{\fcal}{\mathcal{F}}

\newcommand{\ocal}{\mathcal{O}}

\newcommand{\ucal}{\mathcal{U}}

\DeclareMathOperator{\Vol}{Vol}
\DeclareMathOperator{\Op}{Op}

\newcommand{\dbar}{\bar\partial}
\newcommand{\ddbar}{\partial\dbar}
\newcommand{\wt}[1]{\widetilde{#1}}
\newcommand{\half}{{\frac{1}{2}}}

\newtheorem{theo}{\sc Theorem}[section]

\newtheorem{defn}[theo]{\sc Definition}
\newtheorem{lem}[theo]{\sc Lemma}
\newtheorem{prop}[theo]{\sc Proposition}
\newtheorem{rem}[theo]{\sc Remark}

\newtheorem{maintheo}{\sc Theorem}

\title[Log-scale equidistribution of nodal sets]{Log-scale equidistribution of nodal sets in Grauert tubes}

\author{Robert Chang and Steve Zelditch}

\email{hchang@math.northwestern.edu}
\email{zelditch@math.northwestern.edu}

\address{Department of Mathematics, Northwestern University,
Evanston IL,  60208-2730, USA}

\thanks{Research partially supported by NSF grants DMS-1541126 and DMS-1810747}

\begin{document}
\begin{abstract}
Let $M_{\tau_0}$ be the Grauert tube (of some fixed radius $\tau_0$) of a compact, negatively curved, real analytic Riemannian manifold $M$ without boundary. Let $\phi_\lambda$ be a Laplacian eigenfunction on $M$ of eigenvalues $-\lambda^2$ and let $\phi_\lambda^\C$ be its holomorphic extension to $M_{\tau_0}$. In this article, we prove that on $M_{\tau_0} \setminus M$, there exists a dimensional constant $\alpha > 0$ and a full density subsequence $ \{\lambda_{j_k}\}_{k=1}^{\infty}$ of the spectrum for which the  masses of the complexified eigenfunctions $\phi_{\lambda_{j_k}}^\C$ are asymptotically equidistributed at length scale $(\log \lambda_{j_k})^{-\alpha}$. Moreover,  the complex zeros of $\phi_{\lambda_{j_k}}^\C$ also become equidistributed on this logarithmic length scale.

\end{abstract} 

\maketitle

\setcounter{tocdepth}{3}

\section{Introduction}

Let $(M^n,g)$ be a compact, negatively curved, real analytic Riemannian manifold without boundary. 
Let $\Delta = \Delta_g$ be the (negative) Laplacian. We denote by $\{\phi_j\}_{j=0}^{\infty}$  an orthonormal basis of eigenfunctions:
\begin{equation}
(\Delta + \lambda^2_j)\phi_j = 0,\end{equation}
where (as usual) eigenvalues are enumerated in increasing order $0 = \lambda_0 < \lambda_1 \le \lambda_2 \le \dotsb \uparrow \infty$.
To date, no distribution law is known for real nodal sets of Laplacian eigenfunctions on $M$. But, in the manner of \cite{Zelditch07complex}, we are able to obtain a small-scale limit distribution of the complex nodal sets of the analytic continuations of eigenfunctions to the Grauert tube $M_{\tau_0}$ of $M$. By `small-scale' we mean length scales that shrink logarithmically with respect to the frequency parameter $\lambda_j$.  This is the smallest scale to which quantum ergodicity
may be presently localized, as seen in Hezari-Riviere \cite{HezariRiviere16}
and Han \cite{Han15small}. Along individual geodesics, equidistribution of
complex nodal sets is proved down to the scale $\lambda^{-1}$ in \cite{Zelditch14} using quite different arguments.

By a well-known theorem of Bruhat-Whitney, any real analytic manifold $M$ admits a complexification $M_{\C}$ into which it embeds as a totally real submanifold. The metric $g$ on $M$ induces a plurisubharmonic function $\rho$ whose square root $\sqrt{\rho} \colon M_{\C} \rightarrow [0,\infty)$ is called the Grauert tube function. There exists a geometric constant $\tau_0 = \tau_0(M,g) > 0$ so that, for each $\tau \le \tau_0$,  the sublevel set
\begin{equation} \label{EQN:MTAU}
M_\tau := \{\zeta \in M_\C \colon \sqrt{\rho}(\zeta) < \tau\}
\end{equation}
is a strictly pseudo-convex domain in $M_{\C}$. We call $M_\tau$ the \emph{Grauert tube} of $M$ of radius $\tau$.  The $(1,1)$-form $\omega := -i\ddbar \rho$ endows $M_{\tau}$ with a \kahler metric and
$(M, g) \hookrightarrow (M_{\tau}, \omega) $ is an isometric embedding. (The unusual sign convention that makes the \kahler form negative is adopted from \cite{GuilleminStenzel91}.) We write
\begin{equation}\label{EQN:DMUTAU}
d\mu := \omega^n \quad \text{and} \quad d\mu_\tau := \frac{ \omega^n}{d\sqrt{\rho}\mid_{\partial M_\tau}} = \frac{\omega^n}{d\tau}.
\end{equation}
for the \kahler volume form on $M_{\tau}$ and the Liouville surface
measure on $\partial M_{\tau}$, respectively. 
There exists a diffeomorphism $E$, defined in \eqref{eqn:EDEF2}, between $M_\tau$ and the co-ball bundle $B_{\tau}^*M = \{(x,\xi) \in T^*M : \lvert \xi \rvert_{g_x} < \tau\}$. 
The \kahler form $\omega$ on $M_\tau$ is the pullback under $E$ of the
 standard symplectic form on $B_\tau^*M$. Conversely, $E$ endows $B_\tau^* M$ with a complex structure $J_g$ adapted to $g$.  
 Definitions and background  are recalled in Section~\ref{SEC:BACKGROUND}; see also  \cite{GuilleminStenzel92,LempertSzoke91}.

Every eigenfunction $\phi_j$ on $M$ admits an analytic extension $\phi_j^\C$  to the maximal Grauert tube $M_{\tau_0}$. The analytically continued
 eigenfunctions are smooth on the boundaries $\partial M_\tau$ for every $\tau \le \tau_0$. The complex zero set of $\phi_{j}^{\C}$ is the complex hypersurface
\begin{equation}\label{ZlambdaDEF}
\Z_{j} := \{\zeta \in M_{\tau_0}: \phi_j^{\C}(\zeta) = 0\}.
\end{equation}  
The zero sets define currents  $[\Z_{j}] $ of integration
in the sense that for every smooth $(n-1, n-1)$ test form $\eta \in \dcal^{n-1, n-1}(M_{\tau_0})$, we the pairing
\begin{equation}\label{CURRENT}
\langle [\Z_{j}], \eta\rangle : = \int_{\Z_{j}} \eta = \int_{M_{\tau_0}} \frac{i}{2\pi} \ddbar \log \lvert \phi_j^\C \rvert^2 \wedge \eta
\end{equation}
is a well-defined closed current.\footnote{Since $\Z_j$ may be singular, we include background on the last statement in \ref{Appendix:Current}.}
In the special case $\eta = f \omega^{n-1}$, the zero set defines a positive
measure $|\Z_j|$ by 
\begin{equation}
\langle |\Z_j|, f \rangle : = \int_{\Z_j} f \omega^{n-1}, \qquad f \in C(M_{\tau_0}).
\end{equation}

The limit distribution of the zero currents $[\Z_j]$ has been investigated in \cite{Zelditch07complex}. It was shown that on a compact, real analytic, negatively curved manifold,  one has
\begin{equation}\label{EQN:ZELDITCH}
\frac{1}{\lambda_{j_k}}[\Z_{j_k}]  \rightharpoonup \frac{i}{\pi}\ddbar \sqrt{\rho} \quad \text{weakly as currents on $M_{\tau_0}$}
\end{equation}
along a density one subsequence of eigenvalues $\lambda_{j_k}$.
The motivating problem of this article is to obtain a similar convergence theorem  on balls in $M_{\tau_0}\backslash M$ with    logarithmically shrinking radii
of size
\begin{equation}\label{EQN:PARAMETER}
\epsilon(\lambda_j) := (\log \lambda_j )^{-\alpha} \quad \text{for some fixed $\alpha > 0$ to be specified.}
\end{equation}
The parameter $\alpha$ depends only on the dimension, and is independent of the frequency $\lambda_j$.  The resulting log-scale convergence theorems, Theorem~\ref{ZEROSTHintro} and Theorem~\ref{ZEROSTH2}, along with their proofs, are generalizations of those in \cite{ChangZelditch17} in the setting of eigensections of ample line bundles over a compact boundaryless K\"{a}hler manifold, but have several new features.

\begin{maintheo}\label{ZEROSTHintro}
Let $(M,g)$ be a real analytic, negatively curved, compact manifold without boundary. Let $\omega := -i\ddbar \rho$ be the \kahler form on the Grauert tube $M_{\tau_0}$. Assume that
\begin{equation}
0 \le \alpha < \frac{1}{2(3n-1)}, \quad \epsilon(\lambda_j) = (\log \lambda_j)^{-\alpha}.
\end{equation}
Then there exists a full density subsequence of eigenvalues $\lambda_{j_k}$ such that for any $f \in C(M_{\tau_0})$ and for any arbitrary but fixed $\zeta_0 \in M_{\tau_0} \setminus M_{\tau}$, we have
\begin{multline}\label{EQN:ZEROSINTRO}
\left|
\frac{1}{\lambda_{j_k} \epsilon(\lambda_{j_k})^{2n-1}}   \int_{\Z_{j_k} \cap \bcal(\zeta_0, \epsilon(\lambda_{j_k}))}f \omega^{n-1}  -  \frac{1}{\epsilon(\lambda_{j_k})^{2n-1}} \int_{ \bcal(\zeta_0, \epsilon(\lambda_{j_k}))} f \frac{i}{\pi} \ddbar |\Im(\zeta - \zeta_0)|_{g_0} \wedge \omega_0^{n-1} \right|\\
 =   o(1).
\end{multline}
Here, $\omega_0 := -i \ddbar \lvert \Im(\zeta - \zeta_0) \rvert^2_{g_0}$ denotes the flat \kahler form  in local \kahler coordinates centered at $\zeta_0$, with $\lvert \, \cdot \, \rvert_{g_0}$ the Euclidean distance. The $o(1)$ remainder is uniform for any $\zeta_0$ lying in an `annulus' $0 < \tau_1 \le \sqrt{\rho}(\zeta_0) \le \tau_0$.
\end{maintheo}

Theorem~\ref{ZEROSTHintro} is deduced from a rescaled
version given in Theorem~\ref{ZEROSTH2}. The latter theorem is stated using the  holomorphic
dilation introduced in Section~\ref{HOLDIL}. Briefly, define dilation operator $D_{\epsilon(\lambda_j)}^{\zeta_0} \colon \zeta \mapsto \zeta_0 + \epsilon(\lambda_j)(\zeta - \zeta_0)$ in \kahler normal coordinates around $\zeta_0$. The zero currents $[\Z_j]$ on shrinking balls $B(\zeta_0,\epsilon(\lambda_j))$ pulls back to currents  $D^{\zeta_0*}_{\epsilon(\lambda_j)} [\Z_j]$ on a fixed unit ball $B(\zeta_0, 1) \subset \C^n$. The normalizing factors in Theorem~\ref{ZEROSTHintro} arise
from homogeneity and rescaling: 
$\omega^{n-1}, \omega_0^{n-1}$ are homogeneous of degree $2n-2$ and $\frac{i}{\pi}\ddbar |\Im (\zeta-\zeta_0)|_{g_0}$ is homogeneous of degree $1$. The scaling of the nodal
current on the left side is the same as that of its limit current $\frac{i}{\pi}\ddbar |\Im (\zeta-\zeta_0)|_{g_0}$.

\begin{rem}
In the statement of Theorem~\ref{ZEROSTHintro}, the center $\zeta_0$ is arbitrary but fixed in the interior of $M_{\tau_0} \backslash M$  and only the radii of the balls are shrinking. Also, note that $\zeta_0$ must lie away from the totally real submanifold $M$ of $M_{\tau_0}$, or equivalently the zero section $0_M$ of $B_{\tau_0}^* M$. Reasons are discussed in Section~\ref{SINGULAR}. 
\end{rem}

\begin{rem}
The zero sets $\Z_j$ may be singular but
it is known that the singular set of the real nodal set is of real codimension
four (see \ref{Appendix:Current}). For
generic metrics, all of the nodal sets are regular \cite{U76}.
\end{rem}

Knowledge of the log-scale $L^2$ masses of eigenfunctions is required to deduce Theorem~\ref{ZEROSTHintro}. To state the relevant result, we need some more notation:
\begin{equation}\label{Eqn:U}
\Theta_j(\zeta) := \| \phi_j^\C \mid_{\partial M_{\sqrt{\rho}(\zeta)}}\|_{L^2(\partial M_{\sqrt{\rho}(\zeta)})}, \qquad U_j(\zeta) := \frac{\phi_j^\C(\zeta)}{\Theta_j(\zeta)},  \qquad  (\zeta \in M_{\tau_0} \setminus M)
\end{equation}
In words, the normalizing factor $\Theta_j(\zeta)$ is the $L^2$-norm (of the restriction $\phi_j^\C \mid_{\partial M_{\sqrt{\rho}(\zeta)}}$) of $\phi_j^\C$ along the boundary of the Grauert tube of radius $\sqrt{\rho}(\zeta)$. The function $U_j$ is the (the unrestricted) complexified eigenfunction $\phi_j^\C$ normalized by this $L^2$-norm. Finally, let
 \begin{equation} \label{Utaudef}
u^{\tau}_j(Z) := U_j(Z) \mid_{\partial M_\tau} = \frac{\phi_j^\C(Z)\mid_{ \partial M_\tau}}{\| \phi_j^\C \mid_{\partial M_\tau}\|_{L^2( \partial M_\tau)}}, \qquad (Z \in \partial M_\tau, \; 0 < \tau \le \tau_0)
\end{equation}
be the restriction of $U_j$ to the Grauert tube of radius $\sqrt{\rho}(\zeta) = \tau$. (We denote points by $Z$ instead of $\zeta$ when working on a fixed slice $\partial M_\tau$.) The global behavior of $L^2$ masses of $U_{j}$ and $u_{j}^\tau$ are known. Specifically, \cite[Lemma 1.4, Lemma 4.1]{Zelditch07complex} proved the existence of a density
one subsequence $\{\phi_{j_k}\}$ of orthonormal basis such that
\begin{equation}\label{EQN:GLOBALMASS}
\lvert U_{j_k} \rvert^2 \,\omega^n \rightharpoonup  \omega^n \quad \text{and} \quad  \lvert u^{\tau}_{j_k} \rvert^2\, d\mu_{\tau} \rightharpoonup  d\mu_{\tau}
\end{equation}
in the sense of weak* convergence on $C( M_{\tau_0})$ and on $C(\partial M_{\tau})$ for each $0 < \tau \le \tau_0$, respectively. (Recall \eqref{EQN:DMUTAU} for the definitions.) Integrating over $M_{\tau_0}$ (resp.\ $\partial M_\tau$) implies the $L^2$ masses of $U_{j_k}$ (resp.\ $u_{j_k}^\tau$) become equidistributed in all of $M_{\tau_0}$ (resp.\ $\partial M_\tau$). It is not known whether the convergence \eqref{EQN:GLOBALMASS} holds at logarithmic length scales (i.e., simultaneously on all \kahler balls of logarithmically shrinking radii). Luckily, all that is needed for the proof of Theorem~\ref{ZEROSTHintro} is a uniform $L^2$ volume comparison theorem, which we presently state.

\begin{maintheo}\label{PROP:VOLUME} 
Let $(M,g)$ be a real analytic, negatively curved, compact manifold without boundary. Let $\omega := -i \ddbar \sqrt{\rho}$ denote the \kahler form on the Grauert tube $M_{\tau_0}$. Assume that
\begin{equation}
0 \le \alpha < \frac{1}{2(3n-1)}, \quad \epsilon(\lambda_j) = (\log \lambda_j)^{-\alpha}.
\end{equation}
Then there exists a full density subsequence of eigenvalues $\lambda_{j_k}$ such that for arbitrary but fixed $\zeta_0 \in M_{\tau_0} \backslash M$, there is a uniform two-sided volume bound
\begin{equation}\label{EQN:VOLCOMPintro}
c\Vol_\omega(\bcal(\zeta_0,\epsilon(\lambda_{j_k}))) \le \int_{\bcal(\zeta_0,\epsilon(\lambda_{j_k}))} \lvert U_{j_k}\rvert^2 d\mu \le C\Vol_\omega(\bcal(\zeta_0,\epsilon(\lambda_{j_k}))).
\end{equation}
The constants $c, C$ are geometric constants depending only on $\sqrt{\rho}(\zeta_0)$; they are uniform for any $\zeta_0$ lying in an `annulus' $0 < \tau_1 < \sqrt{\rho}(\zeta_0) \leq \tau_0$.
\end{maintheo}

\begin{rem}
Only the lower bound in the statement of Theorem~\ref{PROP:VOLUME}  -- used crucially in a proof by contradiction argument for Proposition~\ref{PROP:LHS} around \eqref{EQN:WDELTA}--\eqref{EQN:RESCALEVOL} -- is needed to imply Theorem~\ref{ZEROSTHintro}.
\end{rem}

Log-scale results of this kind, which we briefly recall in Section~\ref{SSREAL}, were first proved in the real domain by Hezari-Rivi\`{e}re \cite{HezariRiviere16} and X. Han \cite{Han15small}. In the setting of a general compact, negatively curved, \kahler manifold (not necessarily real analytic), an analogous result can be found in \cite[Theorem~2]{ChangZelditch17}.

\begin{rem}
The semi-classical notation $h := \lambda^{-1}$ is also used throughout Section~\ref{CONJSECT}--\ref{LOGQESECT}, in which we write $\delta(h) = \lvert\log h\rvert^{-\alpha} = (\log \lambda)^{-\alpha} = \epsilon(\lambda)$; see \eqref{EQN:H}.
\end{rem}

\subsection{Outline of proof}

Theorem~\ref{PROP:VOLUME} is proved by expressing the $L^2$ mass of $u_j^\tau$ (resp.\ $U_j$) in terms of matrix
elements of Szeg\H o-Toeplitz operators on $\partial M_\tau$ for $0 < \tau \le \tau_0$ (resp.\ Bergman-Toeplitz operators on $M_{\tau_0}$). We show that a certain Poisson-FBI transform conjugates a (smoothed)
characteristic function of the ball $\bcal(\zeta_0,\epsilon(\lambda_j))$ to
a semi-classical pseudodifferential operator acting on $L^2(M)$ whose symbol has the same properties as (but does not coincide with) the small-scale symbols used in \cite{Han15small}. This conjugation allows us to derive Proposition~\ref{PROP:GRAUERTQEBULK}, a variance estimate for matrix elements in the complex domain, by relating it to the known variance
estimate in the real domain of \cite{Han15small}.

Once the variance estimate is proved, the comparability result of Theorem~\ref{PROP:VOLUME} follows the path in \cite{HezariRiviere16, Han15small,
ChangZelditch17}. Namely, one chooses an appropriate covering of $M_{\tau_0}$ and
extracts a subsequence of eigenvalues of density one for which one has
simultaneous asymptotic log-scale QE for the  balls in every cover. 
The balls are `dense enough' that one obtains good upper and lower
bounds for eigenfunction mass in any logarithmically shrinking ball. 

Lastly, to derive 
Theorem~\ref{ZEROSTHintro} from Theorem~\ref{PROP:VOLUME}, we
follow the method of \cite{ShiffmanZelditch99,ChangZelditch17} that uses well-known facts about plurisubharmonic functions. We begin by rewriting the zero current $[Z_j]$ as $\ddbar$ of plurisubharmonic functions using the Poincar\`{e}-Lelong formula \eqref{EQN:PL}. A standard compactness theorem yields the desired result.

\subsection{Singular behavior along the real domain}\label{SINGULAR} 

We briefly discuss the reasons for requiring centers $\zeta_0$ of balls
to lie in $M_{\tau_0} \backslash M$.  

The key tool in studying the 
mass and zeros in the complex domain is the complexified Poisson operator
$P^{\tau} \colon L^2(M) \to \ocal^{\frac{n-1}{4}}(\partial M_{\tau})$ defined in Section~\ref{SEC:PWT}.  By $ \ocal^{-\frac{n-1}{4}}(\partial M_{\tau})$ we mean the Hardy-Sobolev space of boundary values of holomorphic functions in $M_{\tau}$ with the designated Sobolev regularity.  This Hilbert space is the quantization of the symplectic cone $\Sigma_{\tau} \subset T^*(\partial M_{\tau})$ defined in Section~\ref{SEC:SZ}, an $\R_+$-bundle $\Sigma_{\tau} \to \partial M_{\tau}$. The Poisson operator is a homogeneous Fourier integral operator
with positive complex phase adapted to the homogeneous symplectic isomorphism $\iota_\tau \colon  T^* M \setminus 0_M \rightarrow \Sigma_\tau$ of \eqref{iotatau}. 

The homogeneous theory becomes singular along the zero section $0_M$,
or equivalently along the totally real submanifold $M$. This  reflects the fact that
the eigenfunctions $\phi_j$ microlocally concentrate on energy
surfaces $\{|\xi|_g = \lambda_j\}$, the  characteristic variety of $\Delta + \lambda_j^2$. In the semi-classical setting of $h^2 \Delta + 1$ (with $h = \lambda_j^{-1})$, the eigenfunctions concentrate on $S^*M$. The energy level $1$ is arbitrary here and depends on the choice
of constant $C$ in the semi-classical scaling $h_j = C \lambda_j^{-1}$. One may adjust it so that eigenfunctions concentrate on any energy surface $\partial B^*_{\tau} M \simeq \partial M_{\tau}$ with respect to
semi-classical pseudodifferential operators $\Op_{h_j}(a)$. But this  scaling breaks down on the zero section. 

The singularity of the theory along the zero section may be seen in Theorem
\ref{PCONJUGLOG}. When conjugated back to the real domain, the symbols become functions of $|\xi|$ and are singular when $\xi = 0$. It seems that the behavior on the zero section can be studied by using an adapted class of
observables that smoothly interpolates between pseudodifferential operators
when $\tau  =0$ and Toeplitz operators when $\tau > 0$. We hope to clarify this issue in the future.

\subsection{Acknowledgments} We thank the referee for a very careful
reading of the manuscript and for pointing out numerous corrections. We also thank B. Shiffman for contributing to \ref{Appendix:Current}.

\section{Background}\label{SEC:BACKGROUND}

\subsection{Grauert tube and the co-ball bundle}

The readers are referred to \cite{GuilleminStenzel91, GuilleminStenzel92, LempertSzoke91, LempertSzoke01} for the analysis of the complex Monge-Amp\`{e}re equation,  the Grauert tube function, the geometry of Grauert tubes and related topics. Here we provide only a brief summary of some notation and theorems needed for this paper, following \cite{Zelditch07complex, Zelditch14}.

A real analytic manifold $(M,g)$  always possesses a complexification
$M_{\C}$, that is, a complex manifold  of which $M$ is a totally real
embedded submanifold.  Let $\exp_x \colon T^*_xM \rightarrow M$ be the Riemannian exponential map, i.e., $\exp_x \xi = \pi \exp t \Xi_{\lvert \xi \rvert_g^2}$, where $\pi \colon T^*M \rightarrow M$ is the natural projection and $\Xi_{\lvert \xi \rvert_g^2}$ is the Hamiltonian flow of ${\lvert \xi \rvert_g^2}$. The analyticity of $M$ implies that the exponential map admits an analytic extension
\begin{equation}\label{eqn:ECPX}
\exp_x^\C \colon U_x \subset T_x^*M \otimes \C \rightarrow M_\C
\end{equation}
defined in a suitable domain $U_x \subset T_x^*M$. Its restriction to the imaginary axis (that is, the analytic extension in $t$ of $\exp_x(t\xi)$ to imaginary time $t = i$) is denoted by
\begin{equation}\label{eqn:EDEF2}
E \colon B^*_\tau M \rightarrow M_\C, \quad (x,\xi) \mapsto E(x,\xi) := \exp_x^\C(i\xi).
\end{equation}
For all $\tau > 0$ sufficiently small, \eqref{eqn:EDEF2} is a diffeomorphism between the co-ball bundle $B^*_\tau M = \{ (x,\xi) \in T^* M : \lvert \xi \rvert_{g_x} < \tau\}$ and the subset
\begin{equation}
M_\tau := \{ \zeta \in M_\C : \sqrt{\rho}(\zeta) < \tau\} \subset M_\C.
\end{equation}
Here, $\sqrt{\rho}$ is known as the Grauert tube function, and its sublevel set $M_\tau$ is known as the Grauert tube (of radius $\tau$). The restriction $E \mid_{\partial B_\tau^* M}$ of \eqref{eqn:EDEF2} to the co-sphere bundle is a CR holomorphic diffeomorphism between the two strictly pseudo-convex CR manifolds $\partial B_\tau^* M$ and $\partial M_\tau$.

The square $\rho$ of the Grauert tube function is a strictly plurisubharmonic function uniquely determined by two conditions:
\begin{itemize}
\item It is a solution of the 
Monge-Amp\`{e}re equation $(\ddbar \sqrt{\rho} )^n = \delta_{M}$, where
$\delta_M$ is the delta-function on the real manifold $M$ with respect to the volume form $dV_g$;

\item The \kahler form $\omega := -i\ddbar \rho$ restricts to $g$ along $M$.
\end{itemize}
If we write $r(x,y)$ for the Riemannian distance function on $M$, then $r^2(x,y)$
is real analytic in a neighborhood of the diagonal in $M \times M$. It possesses an analytic continuation $r^2(\zeta, \bar{\zeta})$  
for $\zeta \in M_{\C}$ in a sufficiently small neighborhood of the totally real
submanifold $M$. The plurisubharmonic function is related to the Riemannian distance function by
\begin{equation}
\rho(\zeta) = - \frac{1}{4}r^2(\zeta, \bar{\zeta}).
\end{equation}
For the trivial case $M = \R^n$, we have $M_\C = \C^n$ and $\sqrt{\rho}(\zeta) = \sqrt{-\frac{1}{4} (\zeta - \bar{\zeta})^2} = \lvert \Im \zeta \rvert$. More examples are found in \cite{Zelditch07complex}.

\subsection{\texorpdfstring{Szeg\H o}{Szego} projector}\label{SEC:SZ}
Let $\ocal(\partial M_\tau)$ denote the space of CR holomorphic functions on $\partial M_\tau$. We use the notation
\begin{equation}
\ocal^{s + \frac{n-1}{4}}(\partial M_{\tau}) := W^{s + \frac{n-1}{4}}(\partial M_{\tau}) \cap \ocal (\partial M_{\tau})
\end{equation}
for the subspace of the Sobolev space $W^{s +
\frac{n-1}{4}}(\partial M_{\tau})$ consisting of CR
holomorphic functions. The inner product is taken with respect to the Liouville surface
measure \eqref{EQN:DMUTAU}. The Szeg\H{o} projector
\begin{equation}\label{EQN:PITAU}
\Pi_{\tau}  \colon L^2(\partial M_{\tau}) \to \ocal^0(
\partial M_{\tau})
\end{equation}
is the orthogonal projection onto boundary values of
holomorphic function. It is well-known  (cf.\
\cite{BoutetSjostrand76, MelinSjostrand74, GuilleminStenzel92})  that $\Pi_{\tau}$ is a complex
Fourier integral operator of positive type, whose real canonical relation is the
graph of the identity map on the symplectic cone 
\begin{equation}
\Sigma_{\tau} = \{(Z; r d^c\sqrt{\rho}(Z)) \in T^*(\partial M_\tau) : Z \in \partial M_\tau,\; r > 0 \}
\end{equation}
spanned by the contact form $d^c\sqrt{\rho} = -i(\partial - \dbar)\sqrt{\rho}$ on $\partial M_\tau$. Since $\Sigma_{\tau}$ is an $\R_+$-bundle
over $\partial M_\tau$, we can define the symplectic
equivalence of cones:
\begin{equation}\label{iotatau}
\iota_{\tau}  \colon T^*M \setminus 0 \to
\Sigma_{\tau},\quad  \iota_{\tau} (x, \xi) := \left(E \Big(x, \tau
\frac{\xi}{|\xi|}\Big), |\xi| d^c\sqrt{\rho}_{E(x,\tau\frac{\xi}{\lvert \xi \rvert})}\right).
\end{equation}

\subsection{Poisson-wave operator}\label{SEC:PWT}

A key object in our analysis is the Poisson-wave operator
\begin{equation}
P^{\tau} \colon L^2(M) \to \ocal^{\frac{n-1}{4}}(\partial M_{\tau}).
\end{equation}
(Unlike for the Szeg\H{o} projector \eqref{EQN:PITAU}, $\tau$ appears as a superscript here because we will be considering semi-classical Poisson-wave operators, which are denoted by $P^\tau_h$.) The Poisson-wave  operator is obtained from the half-wave operator by analytic extension in the time and spatial variables. Specifically, recall that the half-wave operator is given by $U(t) := e^{it \sqrt{-\Delta}}$. When $t = i\tau$ lies in the positive imaginary axis, $P^\tau := U(i\tau) = e^{-\tau\sqrt{-\Delta}}$ is a complex Fourier integral operator known as the Poisson-wave operator. As discussed in \cite{Boutet78,GuilleminStenzel92,GLS96}, for $0 < \tau \le \tau_0$ and  $y \in M$ fixed, the Poisson kernel $P^\tau(\,\cdot \, , y) = U(i\tau,  \cdot \, , y)$ extends to a holomorphic function on $M_\tau$.

Take for concreteness the wave kernel on $\R^n$ as an example. The Euclidean wave kernel
\begin{equation}
U (t, x, y) = \int_{\R^n} e^{i t |\xi|} e^{i \langle \xi, x
- y \rangle} \, d\xi
\end{equation}
analytically continues  to $(i \tau, x + i p) \in \C_+ \times \C^n$ by the integral formula
\begin{equation}
P^\tau(x+ip,y) = \int_{\R^n} e^{-\tau  |\xi|} e^{i \langle \xi, x - y + i p
\rangle} \,d\xi,
\end{equation}
which converges absolutely for $|p| < \tau$.

On a general Riemannian manifold there exists a similar Lax-H\"{o}rmander parametrix for the wave kernel:
\begin{equation} \label{EQN:LHPARA}
U(t, x, y) = \int_{T^*_y M} e^{i t |\xi|_{y} } e^{i \langle \xi, \exp_y^{-1} (x) \rangle} A(t, x,
y, \xi)\, d\xi,
\end{equation}
where $| \cdot |_{y} $ is the metric norm function at
$y$, and where $A(t, x, y, \xi)$ is a polyhomogeneous amplitude of
order $0$. The
holomorphic extension $x \mapsto \zeta$ to the Grauert tube $M_{\tau_0}$ 
 at time $t = i \tau$ is a Fourier integral operator with complex phase of the form
\begin{equation} \label{EQN:CPXPARA}
P^\tau(\zeta,y) = \int_{T^*_y M} e^{- \tau  |\xi|_{y} } e^{i \langle
\xi, (\exp_y^\C)^{-1} (\zeta) \rangle} A(t, \zeta, y, \xi) \, d\xi.
\end{equation}
The complexified exponential map $\exp_y^\C$ appearing in the phase function of the parametrix above is the local holomorphic extension of the Riemannian exponential map as defined in \eqref{eqn:ECPX}. It is easy to see  that the integral converges absolutely for $\sqrt{\rho}(\zeta) < \tau$. We refer to \cite{Treves80fio, Lebeau13, Zelditch12potential} for proofs and background. The following result is stated by 
Boutet de Monvel \cite{Boutet78}; proofs are given in \cite{Zelditch12potential,Lebeau13}.

\begin{theo}\label{THEO:BDM} Let $\iota_\tau \colon T^* M \setminus 0 \rightarrow \Sigma_\tau$ be the symplectic equivalence defined by \eqref{iotatau}. Then the Poisson-wave operator $P^{\tau} \colon L^2(M)
\to \ocal(\partial M_{\tau} )$ with the parametrix given by \eqref{EQN:CPXPARA} is a  complex Fourier
integral operator of order $- \frac{n-1}{4}$  associated to the positive complex
canonical relation
\begin{equation}
\Gamma := \{(y, \eta, \iota_{\tau} (y, \eta) \} \subset T^*M \times \Sigma_{\tau}.
\end{equation}
Moreover, for any $s$,
\begin{equation}
P^\tau \colon W^s(M) \to {\mathcal O}^{s +
\frac{n-1}{4}}(\partial M_{\tau} )
\end{equation}
is a continuous isomorphism.
\end{theo}

It is helpful to introduce the   framework of {\it adapted Fourier
integral operators}. This notion is defined and discussed in the \cite[Appendix A.2]{BoutetGuillemin81}. If $X, X'$ are two smooth real manifolds, and $\Sigma \subset T^* X \setminus 0$, $\Sigma' \subset T^*X' - 0$ are two symplectic cones, then a   Fourier integral operator $F$ with complex phase  is adapted to a homogeneous
symplectic diffeomorphism  $\chi \colon \Sigma \to \Sigma'$   if  the canonical relation of $F$ is a positive complex canonical relation whose real points consist
of the graph of $\chi$ and if the symbol of $F$ is elliptic. Theorem~\ref{THEO:BDM} 
may be reformulated in this language as follows:
$P^{\tau}$ is a  Fourier integral operator with complex phase of order $- \frac{n-1}{4}$  adapted to 
the symplectic  isomorphism $\iota_{\tau} \colon T^* M \setminus 0 \to \Sigma_{\tau}$ given by \eqref{iotatau}. 
The point of the reformulation is that one may identify the graph of $\iota_{\tau}$
with the graph of $G^{i \tau}$, where $G^t(x, \xi) = |\xi| G^t(x, \frac{\xi}{|\xi|})$ is the homogeneous geodesic flow defined on $T^*M \setminus 0$. Its analytic continuation in $t$ is also homogeneous, so we have
\begin{equation} \label{HOM}
G^{i \tau} (x, \xi) = |\xi| G^{i \tau} \Big(x, \frac{\xi}{|\xi|}\Big).
\end{equation}
 It is observed  in \cite{Zelditch14} that $\iota_{\tau}(y, \eta) = G^{i \tau}(y, \eta)$.  Thus, $G^{i \tau}$ gives a homogeneous
symplectic isomorphism $G^{i \tau} \colon T^* M \setminus 0  \to \Sigma_{\tau}$.

In light of Theorem~\ref{THEO:BDM} and the calculus of FIOs, the operator
\begin{equation} \label{ATAU}
A^{\tau} := (P^{\tau *} P^{\tau})^{-\half} \colon L^2(M) \rightarrow L^2(M).
\end{equation}
is an elliptic, self-adjoint  pseudodifferential operator of order $\frac{n-1}{4}$
with principal symbol $|\xi|^{ \frac{n-1}{4}}$. Equivalently,  $P^{\tau *} P^{\tau}$ is a pseudodifferential
operator of order $-\frac{n-1}{2}$ with principal symbol $|\xi|^{- \frac{n-1}{2}}$. An immediate consequence of Theorem~\ref{THEO:BDM}, \eqref{ATAU} and the symbol calculus of FIOs is the following.

\begin{prop} \label{UNIT}
The operator $V^{\tau}:= P^{\tau} A^{\tau} \colon L^2(M) \to \ocal^0(\partial M_{\tau})$ is unitary (of order 0)
with an approximate left inverse given by $V^{\tau*} A^{\tau} P^{\tau *}$.  Moreover,    $(A^{\tau})^2 P^{\tau *} \colon \ocal^0(\partial M_{\tau}) \to L^2(M)$ is an approximate left inverse to
$P^{\tau}$.
\end{prop}


\subsection{Analytic continuation of eigenfunctions via the Poisson-wave kernel}

Let $\{\phi_j\}$ be an orthonormal basis of Laplacian eigenfunctions on $(M,g)$ with eigenvalue $-\lambda_j^2$. Then the half-wave kernel $U(t,x,y) := e^{it\sqrt{-\Delta}}(x,y)$  admits the eigenfunction expansion
\begin{equation}
U(t,x,y) = \sum_{j=0}^\infty e^{it\lambda_j} \phi_j(x)\overline{\phi_j(y)}.
\end{equation}
It follows that the holomorphic extension to $M_\tau \times M$ of the Poisson kernel is given by
\begin{equation}\label{EQN:COMPLEXPOISSON}
P^{\tau}(\zeta, y) = U(i\tau,\zeta,y) = \sum_{j=0}^\infty e^{-\tau\lambda_j} \phi_j^\C(\zeta)\overline{\phi_j(y)}, \qquad (\zeta, y) \in M_\tau \times M.
\end{equation}
We therefore obtain a formula for the analytic extension $\phi_j^\C$ of an eigenfunction $\phi_j$ to the Grauert tube. Specifically, if $Z \in \partial M_\tau$ (so in particular $\sqrt{\rho}(Z) = \tau$), then
\begin{equation}\label{EQN:CPXEF}
\phi_j^\C(Z) = e^{\tau \lambda_j} (P^\tau \phi_j)(Z) = e^{\sqrt{\rho}(Z) \lambda_j} (P^\tau \phi_j)(Z), \qquad Z \in \partial M_\tau.
\end{equation}

\subsection{\texorpdfstring{Szeg\H{o}}{Szego}-Toeplitz multiplication operators}

Let $M_{\tau_0}$ be a Grauert tube of some fixed radius $\tau_0$. For $0 < \tau \le \tau_0$ we consider operators of the form
\begin{equation} \label{TOEPLITZDEF}
\Pi_{\tau} a \Pi_{\tau} \colon \ocal^0(\partial M_{\tau}) \to \ocal^0(\partial M_{\tau}),
\end{equation}
where by an abuse of notation we write $a$ for multiplication by the symbol $a \in C^{\infty}(\partial M_{\tau})$. The operator \eqref{TOEPLITZDEF} is an example of a Szeg\H o-Toeplitz operator.
More generally, such an operator of order $s$ acting on $H^2(\partial M_{\tau})$  is of the form
$\Pi_{\tau} Q \Pi_{\tau}$, with $Q$ a pseudodifferential operator of order $s$. For this article it suffices to take $Q = a$ to be a multiplication operator. A Szeg\H o-Toeplitz operator might be homogeneous
or semi-classical depending on the nature of $Q$.

\subsection{Poisson conjugation of \texorpdfstring{Szeg\H o}{Szego}-Toeplitz operators}

The conjugation of a Toeplitz multiplication operator by the Poisson-wave FIO is  studied in \cite[Lemma 3.1]{Zelditch07complex} and in \cite[Section 4.1]{Zelditch14}

\begin{lem} \label{PSIDO} Let  $a \in C^{\infty}(M_{\tau_0})$ and let $P^{\tau}$ be the Poisson-wave operator defined by \eqref{EQN:CPXPARA}. Then the conjugation
\begin{equation} \label{CONJUG1}
P^{\tau*} \Pi_\tau a \Pi_\tau P^\tau \in  \Psi^{-
\frac{n-1}{2}}(M)
\end{equation}
is a pseudodifferential operator with principal symbol equal to (the homogeneous extension of) $a(x,\xi) |\xi|_g^{- 
\frac{n-1}{2}}$.
Moreover, let $V^\tau$ be the unitary operator defined in Proposition~\ref{UNIT}, then
\begin{equation} \label{CONJUG2}
V^{\tau *} \Pi_\tau a \Pi_\tau V^{\tau} \in \Psi^{0}(M)
\end{equation}
with principal symbol equal to (the homogeneous extension of) $a(x, \xi)$.
\end{lem}

Note that
$$V^{\tau *} \Pi_\tau a \Pi_\tau V^{\tau} = A^{\tau} P^{\tau*} \Pi_\tau a \Pi_\tau P^\tau A^{\tau}, $$
so that the second statement follows from Proposition \ref{UNIT} or   from the first by \eqref{ATAU}.

 \begin{rem} The factors of $\Pi_\tau$ are redundant here because, by
 Theorem~\ref{THEO:BDM},  $P^\tau$
 maps into the range of $\Pi_\tau$. \end{rem}


\section{Balls and dilation in Grauert tubes}\label{BALLDILSECT} 

The purpose of this section is to introduce the balls and local dilation that are relevant to the calculus of pseudodifferential operators with log-scale symbols.

\begin{defn} \label{BALLSDEF}
We  define \kahler balls $\bcal(\zeta_0, \epsilon(\lambda_j))$ in the Grauert tube to be balls with respect to the \kahler metric $\omega =  -i\ddbar \rho$. For reasons discussed in Section~\ref{SINGULAR}, we consider \kahler balls whose centers $\zeta_0 \in M_{\tau_0} \setminus M$ do not lie on the totally real submanifold $M$. The radii $\epsilon(\lambda_j) = (\log \lambda_j)^{-\alpha}$ shrinks logarithmically relative the frequency parameter $\lambda_j$.
\end{defn}

 We also need to introduce local dilation
centered at points $\zeta_0 \in M_{\tau_0}$. When working with holomorphic or plurisubharmonic functions, we always use local holomorphic dilation. But when working with dilated symbols we may use more general dilation that are more convenient. A technical point to address is that the local  dilation does not preserve the family of  \kahler balls. But for centers close enough to the real domain $M$, the metric is almost Euclidean on logarithmically shrinking balls.

\subsection{Holomorphic dilation}\label{HOLDIL} 


Let $\zeta_0 = E(x_0,\xi_0) \in M_{\tau_0}$ be fixed and consider a local \kahler normal coordinate chart around $\zeta_0$ \cite{GriffithsHarris78}. In such a  chart, the  \kahler potential satisfies $\rho(\zeta, \overline{\zeta}) = |\Im(\zeta - \zeta_0)|^2 + O(|\Im(\zeta - \zeta_0)|^4)$, so that $\ddbar \rho = g_0 + O(|\Im(\zeta - \zeta_0)|^2)$, where
$g_0$ is the standard Euclidean Hermitian metric. We denote the unit
ball centered at $\zeta_0$ in this local Euclidean metric by $ B(\zeta_0, 1)$. 

    The local holomorphic dilation of $B(\zeta_0, 1)$  in \kahler normal coordinates $\zeta$  centered at $\zeta_0 \in M_{\tau_0}\setminus M$ is defined by
\begin{equation}\label{EQN:DIACPX}
D_{\epsilon(\lambda)}^{\zeta_0} \colon B(\zeta_0, 1) \rightarrow B(\zeta_0, \epsilon(\lambda)), \qquad \zeta \mapsto \zeta_0 + \epsilon(\lambda)(\zeta - \zeta_0).
\end{equation}

This choice of local dilation is not adapted
to Grauert tube geometry in that sense that the $\epsilon$-dilate of a point in
$\partial M_{\tau}$ is not necessarily a point in $\partial M_{\epsilon \tau}$.
 But
since the metric and tube function are almost Euclidean in shrinking balls
one has constants $c_g, C_g > 0$ so that
\begin{equation}\label{EQN:GEOM}c_g  \epsilon(\lambda) \sqrt{\rho}(\zeta) \leq   \sqrt{\rho}(D_{\epsilon(\lambda)}^{\zeta_0} \zeta) \leq C_g \epsilon(\lambda) \sqrt{\rho}(\zeta) \end{equation}
provided $\sqrt{\rho}(\zeta)$ is small enough. Indeed, it suffices to verify the inequalities for the Euclidean metric, where $\sqrt{\rho}(\zeta) = |\Im \zeta|$
and where $C_g = c_g =1$.

 %

\subsection{Phase space dilation}\label{PhSpD} 
Theorem~\ref{BCONJUGLOG} introduces another
type of dilation, which is more conveniently expressed in terms of the usual cotangent coordinates $(x, \xi)$. 
The dilation in local coordinates centered at $(x_0, \xi_0) \in \partial B^*_\tau M$ is of the form
\begin{equation} \label{OURDIL}
(x,\xi)  \mapsto \bigg(x_0 + \frac{x - x_0 }{\epsilon(\lambda)}, \xi_0 +\frac{\tau \hat{\xi} - \xi_0}{\epsilon(\lambda)} \bigg), \qquad (x_0, \xi_0) \in \partial B^*_\tau M.
\end{equation}
Note that the unit vector $\hat{\xi} := \xi/|\xi|$ is scaled by the parameter $\tau = \lvert \xi_0 \rvert_{x_0}$, with $(x_0,\xi_0)$ the fixed center of dilation.
 
  This is closely related to, but not identical to, the dilation introduced
  in \cite{Han15small}. In that article one fixes a point $(x_0, \xi_0) \in S^*M = \partial B^*_1 M$ in the unit co-sphere bundle and dilates by
$$(x,  \xi) \mapsto
\bigg(x_0 + \frac{x - x_0 }{\epsilon(\lambda)}, \xi_0 +\frac{\hat{\xi} - \xi_0}{\epsilon(\lambda)}\bigg), \qquad (x_0,\xi_0) \in S^* M. $$
Both types of dilation  are  homogeneous in $\xi$. The one essential difference is that in \eqref{OURDIL}, we allow $|\xi_0|_{x_0} = \tau $ and $\tau \hat{\xi}$ to be any positive numbers
bounded away from zero; they need not be the same. Thus, we are not only localizing in the direction of co-vectors but also in their norms.

%

\section{Poisson conjugation of log-scale Toeplitz operators to semi-classical pseudodifferential operators with log-scale symbols}\label{CONJSECT} 

In this section, we generalize the conjugation result of Lemma~\ref{PSIDO} in two ways. On one hand, we let the symbol depend on the frequency $\lambda$, similar to the $\delta(h)$-(micro)localized symbols \eqref{EQN:DELTASYMB} in the Riemannian setting. On the other hand, we consider Bergman-Toeplitz operators, realized as direct integrals of Szeg\H o-Toeplitz operators. We show that conjugation by the FBI transform takes a  decomposable, log-scale
Bergman-Toeplitz operator to a semi-classical pseudodifferential operator with a log-scale symbol.

It is  convenient to introduce the semi-classical parameter
\begin{equation}\label{EQN:H}
h := \lambda^{-1}, \quad h^{-2} E_j = \lambda_j^2, \quad \delta(h) := \lvert \log h \rvert^{-\alpha} = (\log \lambda)^{-\alpha} = \epsilon(\lambda).
\end{equation}
In this semi-classical notation, the Laplacian eigenfunctions satisfy $\Delta \phi_j = h^{-2}E_j \phi_j = \lambda_j^2 \phi_j$.

\subsection{Semi-classical Poisson-wave operator}\label{SCPOISSON} 
The Poisson
kernel \eqref{EQN:CPXPARA} may be realized as a semi-classical Fourier integral operator with the introduction of a semi-classical parameter $h$.
In the Euclidean case, we define the semi-classical Poisson kernel to be

$$P_{h}^{\tau}(x, y) = h^{-n} \int_{\R^n} e^{\frac{i}{h} \langle x - y, \xi \rangle}
e^{- \tau |\xi|/h} \,d \xi. $$
Here, we use the semi-classical Fourier transform
\begin{equation} \label{SCFT} \fcal_{h} u(y) =  h^{-n} \int_{\R^n} e^{-\frac{i}{h} \langle y, \xi \rangle} f(y)\, dy, \end{equation}
to diagonalize $P^{\tau} = e^{- \tau \sqrt{-\Delta}}$.
It is evident that $P^{\tau}_{h} = P^{\tau}$ by changing variables $\xi \to \xi/h$. Indeed, 
$$P^{\tau}_{h} e^{i \langle x, k \rangle/h} = e^{-\tau |k|/h} e^{i \langle x, k \rangle/h} . $$
Thus $P^{\tau}_{h}$  is still the homogeneous Poisson operator $e^{- \tau \sqrt{-\Delta}}$.  

The same change of variables is valid in the manifold setting \eqref{EQN:LHPARA} and we continue to denote the Poisson operator
in semi-classical form by $P_{h}^{\tau}$. The semi-classical version of the zeroth order unitary operator $V^\tau$ from Proposition~\ref{UNIT} is denoted
\begin{equation}
V_h^\tau := P^\tau_h (P^{\tau *}_h P^\tau_h)^{-\frac{1}{2}} \colon L^2(M) \rightarrow \ocal^0(\partial M_\tau).
\end{equation}

\subsection{Log-scale symbols and semi-classical pseudodifferential operators}


Let $0 \le a \le 1$ be a smooth cutoff function that is equal to $1$ on $B(0,1) \subset \C^n$ and vanishes outside $B(0,2) \subset \C^n$. We use \eqref{eqn:EDEF2} to identify $M_{\tau_0}$ with $B_{\tau_0}^*M$. Using local coordinates induced by $\exp_{x_0}^\C \colon T_{x_0}^* M \otimes \C \rightarrow M_\tau$, consider  symbols that, near $(x_0,\xi_0) \in \partial B^*_\tau M$, are locally of the form
\begin{equation}\label{EQN:AEP}
a_{\delta(h)}^{(x_0,\xi_0)} (x,\xi) := a\left(x_0 +\frac{x - x_0}{\delta(h)}, \xi_0 + \frac{\xi - \xi_0}
{\delta(h)}\right).
\end{equation}

Symbols of the type \eqref{EQN:AEP} satisfy the estimate
\begin{equation} \label{LOGSCALE} |D^{\beta} a_{\delta(h)}^{(x_0,\xi_0)} | \leq C_{\beta} \delta(h)^{- \lvert \beta \rvert},  \end{equation}
and are said to belong to the symbol classes $S^0_{\delta(h)}$.  More generally, a function $b \in C^\infty(T^*M)$ belongs to the symbol class $S_{\delta(h)}^k$ if
\begin{equation}\label{EQN:SYMBEST}
\sup_{(x,\xi) \in T^* M}\lvert \partial^\beta_x \partial^\gamma_\xi b \rvert \le C_{\beta,\gamma} \delta(h)^{-\lvert \beta \rvert - \lvert \gamma \rvert}(1 + \lvert \xi \rvert_x^2)^{(k - \lvert \beta \rvert)/2}
\end{equation}
for some constant $C_{\beta,\gamma}$ independent of $h$.

The semi-classical pseudodifferential operator
quantizing a symbol $a$ is defined by the usual local (semi-classical) Fourier transform
formula 
$$\Op_{h}(a)(x, y) : = \frac{1}{(2\pi h)^n} \int_{\R^n} e^{\frac{i}{h} \langle \xi, x - y \rangle} 
a(x, \xi, h) \,d \xi. $$
The quantization of a symbol $b \in S_{\delta(h)}^k$ is denoted by  $\Op_{h}(b) \in \Psi^k_{\delta(h)}$. We refer to \cite{Han15small} for a discussion of the symbol classes $S_{\delta(h)}^k$ and \cite{Zw} for symbol classes and quantizations in general.

\subsection{Semi-classical Poisson conjugation of log-scale Toeplitz operators}

\begin{theo} \label{PCONJUGLOG}
Let $(x_0, \xi_0) \in \partial B_\tau^*M$ be fixed. For symbols $a_{\delta(h)}^{(x_0,\xi_0)} \in C^\infty(M_{\tau_0})$ of the form \eqref{EQN:AEP}, we have
\begin{equation}\label{EQN:POISSONCONJ}
P^{\tau *}_h  \Pi_{\tau}  a_{\delta(h)}^{(x_0,\xi_0)} \Pi_\tau P^\tau_h = \Op_h\left(h^{\frac{n-1}{2}}\lvert \xi \rvert^{-\frac{n-1}{2}} a\bigg( x_0 +\frac{x - x_0 }{\delta(h)}, \xi_0 +\frac{\tau\hat{\xi} - \xi_0}{\delta(h)}\bigg)\right) \in \Psi_{\delta(h)}^{-\frac{n-1}{2}}(M)
\end{equation}
modulo $h \delta(h)^{-2} \Psi_{\delta(h)}^{-\frac{n-1}{2}}(M)$ and
\begin{equation}
V^{\tau *}_h  \Pi_{\tau}  a_{\delta(h)}^{(x_0,\xi_0)} \Pi_\tau V^\tau_h = \Op_h\left(a\bigg(x_0 + \frac{x - x_0 }{\delta(h)}, \xi_0 + \frac{\tau\hat{\xi} - \xi_0}{\delta(h)}\bigg)\right) \in \Psi_{\delta(h)}^{0}(M)
\end{equation}
modulo $h\delta(h)^{-2}\Psi_{\delta(h)}^0(M)$. Note that the $\tau$-scaling affects only $\hat{\xi} := \xi/|\xi|$.
\end{theo}

 \begin{rem} Note that the factors of $\Pi_\tau$ are redundant because $P^\tau$
 maps into the range of $\Pi_\tau$. We prove only \eqref{EQN:POISSONCONJ} as the second conjugation statement may be proved using the first statement and the composition rule for pseudodifferential operators. \end{rem}
 
 \begin{proof}[Proof of Theorem~\ref{PCONJUGLOG}]  The proof is essentially the same as in Lemma~\ref{PSIDO},
 since the dilation has no effect on the properties of the conjugation. Indeed, conjugation by the Fourier integral operator $P^\tau_h$ preserves  the symbol class $S^*_{\delta(h)}$.
  Since $a_{\delta(h)}^{(x_0,\xi_0)}$ is
a function on $\partial M_\tau$, it defines a {\it homogeneous symbol } of
order zero on $\Sigma_\tau$ in the fiber direction. Under conjugation by
$P^\tau_h$ it goes over to  a pseudodifferential operator of order zero on $M$
whose symbol is the transport $a_{\delta(h)}^{(x_0,\xi_0)}(\iota_\tau(x, \xi))$  to $T^* M \setminus 0_M$, with $\iota_\tau$ given  by \eqref{iotatau}. If $\pi_\tau \colon \Sigma_\tau \to \partial M_\tau$
is the natural projection then 
\begin{equation}
\iota_\tau^* a_{\delta(h)}^{(x_0,\xi_0)}(x, \xi)  = a_{\delta(h)}^{(x_0,\xi_0)}(E(x, \tau \hat{\xi})), \qquad \hat{\xi} = \frac{\xi}{|\xi|}.
\end{equation}
For $\tau, \delta(h)$ small enough we may use the Euclidean approximation to the distance function. If we center the local coordinates at $(x_0, \xi_0)$ then the cutoff as a function on $T^*M$ has the form
\begin{equation} \label{PRINSYM} a_{\delta(h)}^{(x_0,\xi_0)}(\iota_\tau(x, \xi)) = a\left( x_0+ \frac{x - x_0 }{\delta(h)},  \xi_0+\frac{\tau \hat{\xi} - \xi_0}{\delta(h)}\right), \qquad \hat{\xi} = \frac{\xi}{|\xi|}. \end{equation}
Thus, $P^{\tau *}_h  \Pi_\tau a_{\delta(h)}^{(x_0,\xi_0)} \Pi_\tau P^\tau_h$ is a homogeneous 
pseudodifferential operator with dilated symbol.

We now provide more details. Since the calculation is local we first provide a proof in the Euclidean case.

\subsubsection{Euclidean case}

Write $Z = x_1 + i \tau p$ with $|p| = 1$ and centering the dilation at $Z_0 = x_0 + i \xi_0$.  We do not assume $\tau = | \xi_0|$. The composition has the form
\begin{multline}
P_{h}^{\tau*}\Pi_\tau  a_{\delta(h)}^{(x_0,\xi_0)}\Pi_\tau P_{h}^\tau(x, y) \\
= h^{-2n}\tau^{n -1} \int_{\R^{n}\times \R^n \times S^{n-1} \times  \R^n}
 e^{\Psi_0/h} a\bigg(x_0+\frac{x_1 - x_0}{\delta(h)}, \xi_0+\frac{ \tau p - \xi_0}{\delta(h)}\bigg)\, d\xi_1 d \xi_2 d\sigma (p) dx_1,
\end{multline}
where $d\sigma(p)$ is the
standard surface area measure on $S^{n-1}$.  The phase is
\begin{equation}\label{PHI0} 
\Psi_0( \xi_1, \xi_2, x_1, p; x, y, \tau) = - \tau (|\xi_1|
+ |\xi_2|) + i\langle \xi_1, x_1 + i \tau p - y\rangle - i \langle
\xi_2, x - (x_1 - i \tau p) \rangle
 \end{equation}

We note that
$$\Re \Psi_0 = - \tau (|\xi_1|
+ |\xi_2|) -\tau  \langle \xi_1 - \xi_2, p \rangle \leq 0 $$
with equality if and only if $ \hat{\xi}_1 = -  \hat{\xi}_2 = \pm p$, that is, the Schwartz kernel integral is of smooth and of order $O(h^{\infty})$.
We absorb the factor apply the complex stationary phase method to the $ dx_1 d \xi_2 d \sigma(p)$ integral. 
The critical point equations for $\Im \Psi$ in $(x_1, \xi_2) $ are
\begin{equation} \label{CPE}
\begin{dcases}
d_{x_1}\Im \Psi_0 = 0 \iff \xi_1 = - \xi_2, \\ 
d_{\xi_2} \Im \Psi_0 = 0 \iff   x_1 = x  
\end{dcases}
\end{equation} The extra $dp $ integral localizes
 at the above point. Since the $dx_1 d \xi_2$ integral has a non-degenerate Hessian,  we may eliminate the $dx_1 d \xi_2$ integrals by stationary phase, obtaining
 a simpler oscillatory integral 
\begin{equation}
h^{-2n + n} \tau^{n -1} \int_{\R^{n} \times S^{n-1}}
 e^{ \Psi_1/h} a\bigg(x_0+\frac{x - x_0}{\delta(h)},  \xi_0+\frac{\tau p - \xi_0 }{\delta(h)}\bigg) d\xi_1   d\sigma (p),
\end{equation}
with
\begin{equation}
\Psi_1( \xi_1, p; x, y, \tau) = -2 \tau |\xi_1| - 2 \tau \langle \xi_1, p \rangle +  i \langle \xi_1, x - y\rangle.
\end{equation}
 Applying the method of stationary phase (steepest descent) to the
 integral over  $S^{n-1}$ gives the critical point equation $p = - \hat{\xi}_1$,
 i.e., the point where the phase is maximal. 
 It follows that
\begin{multline}
P_{h}^{\tau*}\Pi_\tau a^{(x_0,\xi_0)}_{\delta(h)}\Pi_\tau P^{\tau}_h(x, y)  \\
= h^{-2n+n+\frac{n-1}{2}}\tau^{n -1 - \frac{n-1}{2}}  \int_{\R^{n} }
 e^{i \langle \xi_1, x  - y
\rangle/h}   a\bigg(x_0+\frac{x - x_0}{\delta(h)},  \xi_0+\frac{ \tau \hat{\xi}_1 - \xi_0 }{\delta(h)}\bigg) d\xi_1
\end{multline}
modulo terms of order $h \delta(h)^{-2}$ (since each derivative of the symbol pulls out a factor of $\delta(h)^{-1}$).

\subsubsection{General Riemannian manifold}
The proof is similar on any real analytic Riemannian manifold.  In place of the integral over $\R^n \times S^{n-1}$ we now have an integral over
$Z \in \partial M_\tau $ or $(x_1,s  p) \in  \partial B^*_\tau M$ with $|p| = 1$ under the map $Z = E(x_1, s p) $.  Using
the parametrix \eqref{EQN:CPXPARA}, we have
\begin{multline}\label{PH}
P_{h}^{\tau *} \Pi_\tau a^{(x_0,\xi_0)} _{\delta(h)}\Pi_\tau P_{h}^\tau (x, y) \\
= h^{-2n}\tau^{n-1} \int_{T^*_x M \times T_y^* M \times  \partial M_\tau}
 e^{ \Psi/h}  a\bigg(x_0+\frac{x_1 - x_0}{\delta(h)}, \xi_0+\frac{ s p - \xi_0}{\delta(h)}\bigg) A \overline{A}  \, d\xi_1 d \xi_2 d\mu_\tau(Z) 
\end{multline}
with
\begin{equation}
\Psi =  - \tau \left(|\xi_1|_{x}+   |\xi_2|_{y}\right)   + i  \langle \xi_1, (\exp_y^\C)^{-1}(Z) \rangle - i \langle \xi_2, (\exp_{x}^\C)^{-1} (\bar{Z})\rangle.
\end{equation}
 The phase is only well-defined when $Z$ is sufficiently close to $x$ and to
 $y$, but the phase is non-stationary and the integral is exponentially
 decaying otherwise.   The only points for which the integral is not exponentially decaying  are those $Z$ satisfying
$\Im  \langle \xi_1, (\exp_y^\C)^{-1}(Z) \rangle = \tau |\xi_1|$ (and a similar condition holds with $y$ replaced by $x$ and $\xi_1$ replaced by $\xi_2$).
   Note that  $(\exp_x^\C)^{-1} (Z)
\in  U_x \subset  T_x^* M \otimes \C$. 

The critical set $C_{\Psi}$ of the phase is defined by
$$C_{\Psi} = \{(x, y, \tau; \xi_1, \xi_2, Z): d_{\xi_1, \xi_2, Z} \Psi = 0\}. $$
The associated canonical relation is defined by the embedding
\begin{equation} \label{iota}
\iota_{\Psi} \colon C_{ \Psi} \to T^* M \times T^*M, \quad
(x, y, \tau; \xi_1, \xi_2, Z) \to (x, d_x  \Psi, y, - d_y  \Psi). \end{equation}
The composite operator is manifestly a Fourier integral operator with complex phase, and is  a pseudodifferential operator if and only if
$C_{ \Psi} = \Delta_{T^* M \times T^*M}$ (the diagonal).

Let  $Z = E(x_1, \tau p)$. Then the critical point equations are
 \begin{itemize}
\item[(i)] $d_{\xi_1}  \Psi = 0 \iff (\exp_y^\C)^{-1} (Z)=  - i\tau \hat{\xi}_1    \iff  x_1 = y, \; p  = - i \tau \hat{\xi}_2$,
\item[(ii)]  $d_{Z}  \Psi = d_{Z} \left( \langle (\exp_y^\C)^{-1}(Z), \xi_2 \rangle  - (\exp_x^\C)^{-1}(\bar{Z}), \xi_1 \rangle \right) = 0$,
\item[(iii)] $d_{\xi_2}  \Psi = 0 \iff (\exp_x^\C)^{-1} (\bar{Z})=  - i\tau \hat{\xi}_2$. 
\end{itemize}
Equations  (i) and (iii) show that  
\begin{equation} \label{Z} 
Z =  \exp_x^\C  ( i\tau \hat{\xi}_2) = \exp_y^\C  (- i\tau\hat{\xi}_1).\end{equation}
This implies that $Z \in \pi_\tau^{-1}(x) \cap \pi_\tau^{-1}(y)$, where $\pi_\tau \colon \partial M_\tau \to M$. 
Of course, these fibers are disjoint unless $x = y$, so only in that case
does there exist a solution of the critical point equation.  It then follows
that $\hat{\xi_1} = - \hat{\xi}_2$. 

To see that $\xi_1 = -\xi_2$ on the critical point set, we use further use (ii). There  only exists  a solution
of the critical point equations when $x = y$, and then we may write
$Z = u + i v \in T_x^* M\otimes \C$ and  study the restricted critical point equation
$$d_{Z} \Psi = 0 \iff d_{u,v} 
 \left( \langle u + iv, \xi_2 \rangle  - \langle u + iv, \xi_1 \rangle \right) = 0. $$
 Just using $u \in T_x^* M$ already shows that $\xi_1 = \xi_2$ on the critical set.

To calculate \eqref{iota} we may use the Euclidean approximation to the phase  based
at $(x, \xi_1)$ because on $C_\Psi$ only the first order terms in the Taylor expansion 
of $\Psi$ contribute.
But then it is evident that $d_x \Psi = \xi_2 = - d_y  \Psi |_{y = x}
= \xi_1$,  proving that the canonical relation is the diagonal. 

The principal symbol of $P_{h}^{\tau *} \Pi_\tau P_{h}^\tau (x, y) $ is calculated in 
\cite{Zelditch07complex} and the principal symbol of $P_{h}^{\tau *} \Pi_\tau a^{(x_0,\xi_0)} _{\delta(h)}\Pi_\tau P_{h}^\tau(x, y) $ is the same multiplied by the value
of $a^{(x_0,\xi_0)} _{\delta(h)}$ at the critical point. Note that because of the symbol class we are working with, the sub-leading term is of order $h \delta(h)^{-2}$ as each derivative of the symbol pulls out a factor of $\delta(h)^{-1}$. If we use $V_{h}^\tau$ in place of $P_{h}^\tau$ as in Proposition~\ref{UNIT} then the principal symbol is the one stated in Theorem~\ref{PCONJUGLOG}.
\end{proof}

\subsection{Comparison of symbols}
We note that symbols of the form \eqref{PRINSYM}  are not quite the same as the log-scaled symbols $a_{z_0}^b (x,\xi; h) 
$ of \eqref{EQN:DELTASYMB} considered in \cite{Han15small}. However, as long as $(x_0, \xi_0)$ are fixed at a positive distance from the real domain $M$, the same symbol estimates
\eqref{LOGSCALE}
are valid. Also note that it is not necessary to multiply by a cutoff $\phi(|\xi|)$  to
$S^*M$ since the cutoff $a_{z_0}^b (x,\xi; h) 
$  is supported in a shrinking \kahler ball around $E(x_0, \xi_0)$. In fact,
we define the sequence $h_j$ so that eigenfunctions concentrate
on the energy surface $\partial M_{\tau_0}$ with $|\xi_0|_{x_0} = \tau_0$. There
is no difficulty as long as $\tau_0 > 0$. We continue to use the notation
$\Op_{h}(a)$ for semi-classical pseudodifferential operators with 
symbols of the form \eqref{PRINSYM}.

\section{Decomposable Poisson-FBI transform and Bergman-Toeplitz operators}

In this section we introduce a Poisson FBI transform taking $L^2(M)$
to a weighted Hilbert space of holomorphic functions on $M_{\tau}$ rather than to CR-holomorphic functions on $\partial M_\tau$.  As explained in Section \ref{DIRINT}, it  is defined in a novel way  by a direct integral of Poisson transforms $P^s$,  and therefore all of its main properties
flow from those established above for the Poisson kernel. The main result is
the conjugation Theorem \ref{BCONJUGLOG}. 

\subsection{\label{DIRINT}Weighted Bergman space and  Poisson-FBI transform}

The Poisson kernel endows $\ocal^0(M_{\tau})$ with a plurisubharmonic weight $e^{- \sqrt{\rho}/h}$. We define
\begin{equation} \label{WB}    A^2(M_{\tau},  h^{-\frac{n-1}{2}} e^{- 2 \sqrt{\rho}/h } d \mu)  \end{equation}
to be the Hilbert space of holomorphic functions on $M_{\tau}$ that lie in $L^2(M_{\tau}, e^{- 2 \sqrt{\rho}/h} d \mu) . $
It is isometric to the Hilbert space
\begin{equation} H_{\sqrt{\rho}} := \{f   h^{-\frac{m-1}{4}} e^{- \sqrt{\rho}/h } : f \in A^2(M_{\tau}\} \subset L^2(M_{\tau}, d\mu) \end{equation} endowed with the inner product of $L^2(M_{\tau}, d\mu) $.

It is useful to regard $H_{\sqrt{\rho}}$ as a direct integral 
$$H_{\sqrt{\rho}} = \int_{[0, \tau_0]}^{\oplus} H^2(\partial M_\tau) \,d\tau $$of Hilbert spaces
$H^2(\partial M_\tau)$. Here, 
  $\int_{[0, \tau_0]}^{\oplus} H^2(\partial M_\tau) \,d\tau$ denotes  the space of $L^2$ sections $f(\tau) \in H^2(\partial M_\tau)$
of the Hilbert bundle, and the direct integral formula follows from
  Fubini's theorem,
$$\|f\|^2 = \int_0^{\tau_0}  \left(\int_{\partial M_\tau} |f(Z)|^2\, d\mu_\tau(Z) \right) d\tau. $$

        We then define the `moving Poisson operator' or FBI transform by 
\begin{equation}
 T_{h} f(\zeta) = P^{\sqrt{\rho}(\zeta)} f(\zeta) = \int_M P^{\sqrt{\rho}(\zeta)}(\zeta, y) f(y)\, dV(y), \qquad \zeta \in M_{\tau_0}.  \end{equation}
We claim that  $T_{h} \colon L^2(M) \to H_{\sqrt{\rho}} $
is a unitary operator.  To see this, we 
use that $P^\tau$ is unitary from $L^2(M)$ to each integrand, and observe that
$$T_{h}  = \int_{[0, \tau_0]}^{\oplus} P^\tau_h \,d\tau$$ is the direct integral of a family of unitary operators index by $\tau$. 
 
 \subsection{FBI conjugation theorem}

Next we define Bergman-Toeplitz operators.  For $a \in C^\infty(M_{\tau_0})$ define 
\begin{equation}\label{EQN:BTOP}
\wt{\Op}_h(a) = \int_{[0, \tau_0]}^{\oplus} \Pi_\tau (a |_{\partial M_\tau}) \Pi_\tau \,d\tau.
\end{equation}
Implicitly $H^2(\partial M_\tau) \perp H^2(\partial M_\sigma)$ if
$\tau \neq \sigma$. This is a decomposable operator.

\begin{theo} \label{BCONJUGLOG}
For symbols $a_{\delta(h)}^{(x_0,\xi_0)} \in C^\infty(M_{\tau_0})$ of the form \eqref{EQN:AEP}, we have
\begin{multline}\label{EQN:FBICONJ}
T_{h}^* \wt{\Op}_{h} (a_{\delta(h)}^{(x_0,\xi_0)}) T_{h} \\
= \Op_{h}\left(\int_0^{\tau_0} h^\frac{n-1}{2}\lvert \xi \rvert^{-\frac{n-1}{2}}a\bigg(x_0+\frac{x - x_0}{\delta(h)}, \xi_0+\frac{\tau\hat{\xi} - \xi_0}{\delta(h)}\bigg)d\tau\right) \in \Psi^{-\frac{n-1}{2}}_{\delta(h)}(M).
\end{multline}
\end{theo}
Note that \eqref{EQN:FBICONJ} follows from \eqref{EQN:POISSONCONJ} thanks to the identity
\begin{equation}
T_{h}^* \wt{\Op}_h(a) T_{h} = \int_0^{\tau_0} P^{\tau*}_h \wt{\Op}_{h}(a) P^\tau_h\, d\tau.
\end{equation}
Indeed, a multiplication operator is automatically decomposable and the Schwartz kernel is
\begin{equation}
\int_{M_{\tau_0}} P_h^*(x, \zeta) a(\zeta) P_h(\zeta, y) \,d\mu(\zeta) = \int_0^{\tau_0} \left(\int_{\partial M_\tau} P^{\tau*}_h(x, Z) a(Z) P^\tau_h(Z, y) \, d \mu_\tau(Z)\right) d\tau.
\end{equation}
By Theorem~\ref{PCONJUGLOG}, each integrand of the $d\mu_\tau(Z)$ integral in the expression above is a semi-classical
pseudodifferential operator by \eqref{EQN:POISSONCONJ}. The entire $d\tau$ integral is therefore an integral of an analytic family (in $\tau$) of semi-classical pseudodifferential operators on $M$ with the prescribed principal symbol.

\section{Log-scale quantum ergodicity in the real domain}\label{SSREAL}

A key part of our analysis is to relate log-scale quantum variance estimates in the complex domain to those in the real domain, and reduce variance estimates to the  small-scale quantum ergodicity results on negatively curved Riemannian manifolds due to Hezari-Rivi\`{e}re~\cite{HezariRiviere16} and Han~\cite{Han15small}.  We briefly review their results in preparation 
for the next section.

As before, let  $\delta(h) = \lvert \log h \rvert^{-\alpha}$, with the semi-classical parameter given by \eqref{EQN:H}. Consider compactly supported smooth functions that, near $z_0 = (x_0,\xi_0) \in S^*M$, can be locally expressed as
\begin{equation}\label{EQN:DELTASYMB}
a_{z_0}^b (x,\xi; h) := b\bigg(x_0 +\frac{x - x_0}{\delta(h)}, \xi_0 +\frac{\hat{\xi} - \hat{\xi}_0}{\delta(h)}\bigg)\varphi(\lvert \xi \rvert_x) \in S_{\delta(h)}^0,
\end{equation}
where $b \in C^\infty_c(\R^n \times \R^{n-1})$ is some compactly supported smooth function and where  $\varphi \in C^\infty_c((1-1/2, 1+ 1/2))$ is a smooth cutoff function that is identically 1 on $(1-1/4, 1+ 1/4)$.\footnote{There is a misprint in \cite{Han15small} where the support is said to be $(-\half, \half)$ around the zero section $0_M$. In fact, it needs to be around $S^*M$.} It is easy to see that such a function belongs to the symbol class $S_{\delta(h)}^0$ by verifying the symbol estimate \eqref{EQN:SYMBEST}. The following results pertains to $\delta(h)$-microlocalized symbols \eqref{EQN:DELTASYMB}.

\begin{theo}[{\cite[Theorem~1.6]{Han15small}}] \label{VARTH}
Let $(M^n,g)$ be negatively curved (not necessarily real analytic). Let
\begin{equation}
0 < \alpha < \frac{1}{2(2n-1)}, \; 0 \le \beta < 1-2\alpha(2 n-1) \quad \text{or} \quad \alpha = 0, \; \beta = 1.
\end{equation}
Set $\delta(h) = \lvert \log h \rvert^{-\alpha}$. Then for any orthonormal basis $\{\phi_j\}$ of $h^2\Delta$, we have
\begin{equation}
h^{n-1} \sum_{E_j \in [1, 1+h]} \left\lvert \langle \Op_h(a_{z_0}^b) \phi_j, \phi_j \rangle - \dashint_{S^*M}a_{z_0}^b\,d\mu_L \,\right\rvert^2 = \ocal(\delta(h)^{2(2n-1)} \lvert \log h \rvert^{-\beta}).
\end{equation}
Here, $\Op_h$ is a suitable semi-classical quantization, and $d\mu_L$ is the Liouville measure.
\end{theo}

A covering argument using balls of inverse logarithmic radii implies the next volume comparison result.
\begin{theo}[{\cite[Corollary~1.9]{Han15small}; see also \cite[Lemma~3.1]{HezariRiviere16}}] \label{MAINHAN}
Let $(M^n,g)$ be negatively curved (not necessarily real analytic). Let
\begin{equation}
0 < \alpha < \frac{1}{3n} \quad \text{and} \quad r(\lambda) = (\log \lambda)^{-\alpha}.
\end{equation}
Ten, there exists a full density subsequence such that
\begin{equation}
c \Vol(B(x,r_{j_k})) \le \int_{B(x,r_{j_k})} \lvert\phi_{j_k} \rvert^2 \,dV \le C \Vol(B(x,r_{j_k}))
\end{equation}
uniformly for all $x \in M$, where $c, C > 0$ depends only on $(M,g)$.
\end{theo}

\begin{rem} \label{SYMBREM}
An important technical point for this article is that the proofs
of the theorems hold for symbols in $S_{\delta(h)}^0$; the precise form of $a_{z_0}^b$ is not relevant.
\end{rem}

\section{Log-scale quantum ergodicity in Grauert tubes: Proof of Theorem~\ref{PROP:VOLUME}}\label{LOGQESECT}

We introduce some notation. Let
\begin{equation}\label{THETA}
\Theta_j(\zeta) := \big\| \phi^\C_j \mid_{\partial M_{\sqrt{\rho}(\zeta)}}\big\|_{L^2(\partial M_{\sqrt{\rho}(\zeta)})}
\end{equation}
denote the $L^2$-norm of $\phi_j^\C$ restricted to the boundary of the Grauert tube of radius $\sqrt{\rho}(\zeta)$. Let
\begin{equation}\label{BIGU}
U_j(\zeta) := \frac{\phi_j^\C(\zeta)}{\Theta_j(\zeta)}
\end{equation}
denote the normalized complexified eigenfunction. We will also consider its restriction to $\partial M_\tau$ for each $0 < \tau \le \tau_0$ fixed:
\begin{equation}\label{SMALLU}
u_j^\tau(Z) := U_j(Z) \mid_{\partial M_\tau} = \frac{\phi_j^\C(Z) \mid_{\partial M_\tau}}{\big\|\phi^\C_j \mid_{\partial M_{\tau}}\big\|_{L^2(\partial M_{\tau})}}, \qquad (Z \in \partial M_\tau).
\end{equation}
Note that the denominator in \eqref{SMALLU} is a constant (depending on $\tau$), and the numerator is a CR-holomorphic function on $\partial M_\tau$.

\subsection{Variance estimates in Grauert tubes}\label{VARSECT} 
We begin with a log-scale variance estimate for symbols on $\partial M_\tau$, which parallels \cite[Theorem 4]{ChangZelditch17}. Using the $E$ map \eqref{eqn:EDEF2} to identify $B^*_{\tau_0} M$ with $M_{\tau_0}$, we henceforth write 
\begin{equation}
a_{\delta(h)}^{\zeta_0} := a_{\delta(h)}^{(x_0,\xi_0)} \in C^\infty(M_{\tau_0}), \qquad \zeta_0 = E(x_0,\xi_0)
\end{equation}
for small-scale symbols of the form \eqref{EQN:AEP}. We write $Z$ in place of $\zeta$ when restricting to the boundary $\partial M_\tau$, so for instance
\begin{equation}
a^{\zeta_0}_{\delta(h)}(\zeta)\mid_{\partial M_\tau} = a^{\zeta_0}_{\delta(h)}(Z), \qquad Z \in \partial M_\tau.
\end{equation}

\begin{prop}\label{PROP:GRAUERTQE}
Let $(M^n,g)$ be negatively curved and real analytic. Let
\begin{equation}
0 < \alpha < \frac{1}{2(2n-1)}, \; 0 \le \beta < 1-2\alpha(2n-1)  \quad \text{or} \quad \alpha = 0, \; \beta = 1.
\end{equation}
Set $\delta(h) = \lvert \log \delta \rvert^{-\alpha}$ as in \eqref{EQN:H}. Let $\{\phi_j\}$ be an orthonormal basis of eigenfunctions for $\Delta$. Then for every $0 < \tau \le \tau_0$ and every
$\zeta_0 \in M_{\tau} \backslash M$, we have
\begin{multline}
h^{n-1} \sum_{E_j \in [1,1+h]}\left\lvert\int_{\partial M_\tau} a^{\zeta_0}_{\delta(h)}(Z) \lvert u_j^\tau(Z) \rvert^2\,d\mu_\tau(Z) - \frac{1}{\mu_\tau(\partial M_\tau)} \int_{\partial M_\tau} a^{\zeta_0}_{\delta(h)}(Z)\,d\mu_\tau \right\rvert^2\\
 = \ocal( \delta(h)^{2(2n-1)} \lvert\log h\rvert^{-\beta}).
\end{multline}
The remainder is uniform for any $\zeta_0$ in an `annulus' $0 < \tau_1 \le \sqrt{\rho}(\zeta_0) \le \tau_0$.
\end{prop}

\begin{proof}
We  use  Proposition~\ref{PCONJUGLOG} to  transport matrix elements on  $\partial M_\tau $ to matrix elements of  pseudodifferential operators on $L^2(M)$. Since the restriction  $\phi_{h}^\C (Z)$ to $\partial M_\tau$ is a CR-holomorphic function, it satisfies $\Pi_\tau \phi_j^\C (Z) = \phi_j^\C (Z)$. Moreover, $e^{-2\sqrt{\rho}(Z)/h} = e^{-2\tau/h}$ on $\partial M_\tau$.  Therefore,
\begin{align}
\int_{\partial M_\tau} a^{\zeta_0}_{\delta(h)}(Z) \lvert u_j^\tau(Z) \rvert^2 \,d\mu_\tau(Z) &= \|\phi_j^\C\|_{L^2(\partial M_\tau)}^{-2}\left\langle a^{\zeta_0}_{\delta(h)} \Pi_\tau \phi_j^\C, \Pi_\tau \phi_j^\C \right\rangle_{L^2(\partial M_{\tau})}\\
&= e^{2\tau/h}  \|\phi_j^\C\|_{L^2(\partial M_\tau)}^{-2} \left\langle a_{\delta(h)}^{\zeta_0} \Pi_\tau P^\tau_h \phi_j, \Pi_\tau P^\tau_h\phi_j \right\rangle_{L^2(M)}\\
&= \frac{\langle P^{\tau*}_h \Pi_\tau a_{\delta(h)}^{\zeta_0} \Pi_\tau P^\tau_h \phi_j, \phi_j \rangle_{L^2(M)}}{\langle P^{\tau*}_h \Pi_\tau P^\tau_h \phi_j, \phi_j \rangle_{L^2(M)}}.\label{EQN:REWRITE}
\end{align}
The last equality follows from setting $a^{\zeta_0}_{\delta(h)} \equiv 1$, which implies
\begin{equation}
1 = e^{2\tau/h}  \|\phi_j^\C\|_{L^2(\partial M_\tau)}^{-2} \left\langle  P^{\tau*}_h\Pi_\tau P^\tau_h \phi_j, \phi_j \right\rangle_{L^2(M)}.
\end{equation}
By Proposition~\ref{PCONJUGLOG},  $P^{\tau*}_h \Pi_\tau a^{\zeta_0}_{\delta(h)} \Pi_\tau P^\tau_h$ is an $h$-pseudodifferential operator with principal symbol
\begin{equation}
h^{\frac{n-1}{2}}\lvert \xi \rvert^{-\frac{n-1}{2}} a\bigg(x_0+\frac{x-x_0}{\delta(h)}, \xi_0+\frac{\tau \hat{\xi} - \xi_0}{\delta(h)}\bigg).
\end{equation}
By taking $a^{\zeta_0}_{\delta(h)} \equiv 1$ in Theorem~\ref{PCONJUGLOG}, the denominator $P^{\tau*}_h \Pi_\tau P^\tau_h = P^{\tau*}_hP_h^\tau$ is found to be an $h$-pseudodifferential operator with principal symbol $h^{\frac{n-1}{2}}\lvert \xi \rvert^{-\frac{n-1}{2}}$. The quotient \eqref{EQN:REWRITE} may be rewritten using Proposition~\ref{PCONJUGLOG}:
\begin{align} 
\int_{\partial M_\tau} a^{\zeta_0}_{\delta(h)}(Z)& \lvert u_j^\tau(Z) \rvert^2 \,d\mu_\tau(Z)\\
&= \frac{\left\langle \Op_h\left(h^{\frac{n-1}{2}}\lvert \xi \rvert^{-\frac{n-1}{2}} a\left(x_0+\frac{x-x_0}{\delta(h)}, \xi_0+ \frac{\tau \hat{\xi} - \xi_0}{\delta(h)}\right)\right)\phi_j, \phi_j \right\rangle_{L^2(M)} + \ocal(h\delta(h)^{-2})}{\left\langle \Op_h\left(h^{\frac{n-1}{2}}\lvert \xi \rvert^{-\frac{n-1}{2}}\right) \phi_j, \phi_j \right\rangle_{L^2(M)} + \ocal(h\delta(h)^{-2})}\\ 
&= \left \langle \Op_h\left(a\bigg(x_0+\frac{x-x_0}{\delta(h)}, \xi_0+\frac{\tau \hat{\xi} - \xi_0}{\delta(h)}\bigg)\right)\phi_j, \phi_j\right \rangle_{L^2(M)} + \ocal(h\delta(h)^{-2})\\ \label{INNER}
&= \left\langle V^{\tau*}_h\Pi_\tau a_{\delta(h)}^{\zeta_0}\Pi_\tau V^\tau_h \phi_j, \phi_j \right\rangle_{L^2(M)} + \ocal(h\delta(h)^{-2}).
\end{align}
As noted in Remark~\ref{SYMBREM},   Theorem~\ref{VARTH}
 applies to symbols in the symbol class $S^0_{\delta(h)}$. But $V^{\tau*}_h\Pi_\tau a_{\delta(h)}^{\zeta_0}\Pi_\tau V^\tau_h \in \Psi_{\delta(h)}^0(M)$, so the proof is complete.
\end{proof}

\begin{prop}\label{PROP:GRAUERTQEBULK}
With the same notation and assumptions as in Proposition~\ref{PROP:GRAUERTQE}:  For every $\zeta_0 \in M_\tau \setminus M$ and $a^{\zeta_0}_{\delta(h)}$, we have
\begin{multline}
h^{n-1} \sum_{E_j \in [1,1+h]}\left\lvert  \int_{ M_{\tau_0} } a^{\zeta_0}_{\delta(h)} (\zeta) \lvert U_{j}(\zeta) \rvert^2\,d\mu(\zeta) - \int_0^{\tau_0} \! \int_{\partial M_\tau}  \frac{a^{\zeta_0}_{\delta(h)} (Z) }{\mu_{\tau}(\partial M_\tau)}  \,d\mu_\tau(Z) d\tau \right\rvert^2\\
= \ocal(\delta(h)^{4n} \lvert\log h \rvert^{-\beta}).
\end{multline}
The remainder is uniform for any $\zeta_0$ in an `annulus' $0 < \tau_1 \le \sqrt{\rho}(\zeta_0) \le \tau_0$.
\end{prop}

\begin{proof} 
Rewrite the integral over $M_{\tau_0}$ as an iterated integral:
\begin{equation}
\int_{ M_{\tau_0} } a^{\zeta_0}_{\delta(h)} (\zeta) \lvert U_{j}(\zeta) \rvert^2\,d\mu(\zeta) = \int_0^{\tau_0} \! \int_{\partial M_\tau}  a^{\zeta_0}_{\delta(h)} (Z) \lvert u_j^\tau (Z) \rvert^2\,d\mu_\tau(Z)d\tau.
\end{equation}
We make two observations. First, for the outer integral it suffices to integrate over $\tau \in [{\sqrt{\rho}(\zeta_0) - 2\delta(h)}, {\sqrt{\rho}(\zeta_0)+2\delta(h)} ]$ thanks to the choice \eqref{EQN:AEP} of symbols. Second, the inner integral may be replaced by matrix elements of $V_h^{\tau *} \Pi_\tau a_{\delta(h)}^{\zeta_0} \Pi_\tau V_h^{\tau}$ at the cost of $\ocal(h\delta(h)^{-2})$ in light of \eqref{INNER}:
\begin{align}
\int_{ M_{\tau_0} } a^{\zeta_0}_{\delta(h)} (\zeta) \lvert U_{j}(\zeta) \rvert^2\,d\mu(\zeta) &= \int_{\sqrt{\rho}(\zeta_0) - 2\delta(h)}^{\sqrt{\rho}(\zeta_0)+2\delta(h)} \left(\left\langle V^{\tau*}_h\Pi_\tau a_{\delta(h)}^{\zeta_0}\Pi_\tau V^\tau_h \phi_j, \phi_j \right\rangle d\tau + \ocal(h\delta(h)^{-2})\right)\\
&= \int_{\sqrt{\rho}(\zeta_0) - 2\delta(h)}^{\sqrt{\rho}(\zeta_0)+2\delta(h)} \left\langle V^{\tau*}_h\Pi_\tau a_{\delta(h)}^{\zeta_0}\Pi_\tau V^\tau_h \phi_j, \phi_j \right\rangle d\tau + \ocal(h\delta(h)^{-1}).
\end{align}
We now subtract $\int_0^{\tau_0} \! \int_{\partial M_\tau}  \frac{a^{\zeta_0}_{\delta(h)} (Z) }{\mu_{\tau}(\partial M_\tau)}  \,d\mu_\tau(Z) d\tau$ from both sides of the equality and then square both sides. The error is then of order $h^2 \delta(h)^{-2}$, which we move to the left-hand side of the equality to conserve space:
\begin{align}
\left\lvert  \int_{ M_\tau } \right.  & \left. a^{\zeta_0}_{\delta(h)} (\zeta)  \lvert U_{\lambda_j}(\zeta) \rvert^2\,d\mu(\zeta) - \int_0^{\tau_0} \! \int_{\partial M_\tau}   \frac{a^{\zeta_0}_{\delta(h)} (Z) }{\mu_{\tau}(\partial M_\tau)}  \,d\mu_\tau(Z) d\tau \right\rvert^2+ \ocal(h^2\delta(h)^{-2}) \\
&= (4\delta(h))^2\left\lvert \int_{ \sqrt{\rho}(\zeta_0) - 2\delta(h)}^{\sqrt{\rho}(\zeta_0) + 2\delta(h)} \left( \left\langle V^{\tau*}_h\Pi_\tau a_{\delta(h)}^{\zeta_0}\Pi_\tau V^\tau_h \phi_j, \phi_j \right\rangle  - \int_{\partial M_\tau} \frac{a^{\zeta_0}_{\delta(h)} (Z) }{\mu_{\tau}(\partial M_\tau)}\,d\mu_\tau(Z)\right) \frac{d\tau}{4\delta(h)} \right\rvert^2\\
&\le  (4\delta(h))^2 \int_{ \sqrt{\rho}(\zeta_0) - 2\delta(h)}^{\sqrt{\rho}(\zeta_0) + 2\delta(h)} \left\lvert \left\langle V^{\tau*}_h\Pi_\tau a_{\delta(h)}^{\zeta_0}\Pi_\tau V^\tau_h \phi_j, \phi_j \right\rangle  - \int_{\partial M_\tau} \frac{a^{\zeta_0}_{\delta(h)} (Z) }{\mu_{\tau}(\partial M_\tau)}\,d\mu_\tau(Z) \right\rvert^2 \frac{d\tau}{4\delta(h)} \\
&= 4\delta(h) \int_{ \sqrt{\rho}(\zeta_0) - 2\delta(h)}^{\sqrt{\rho}(\zeta_0) + 2\delta(h)} \left\lvert \left\langle V^{\tau*}_h\Pi_\tau a_{\delta(h)}^{\zeta_0}\Pi_\tau V^\tau_h \phi_j, \phi_j \right\rangle  - \int_{\partial M_\tau} \frac{a^{\zeta_0}_{\delta(h)} (Z) }{\mu_{\tau}(\partial M_\tau)}\,d\mu_\tau(Z) \right\rvert^2 d\tau.
\end{align}
For the inequality we used that $\frac{d\tau}{4\delta(h)}$ is a probability measure on the interval $[\sqrt{\rho}(\zeta_0) - 2\delta(h), \sqrt{\rho}(\zeta_0) + 2\delta(h)]$, so Jensen's inequality  applies. Performing the Ces\`aro sum and using  Proposition~\ref{PROP:GRAUERTQE}, we find
\begin{align}
h^{n-1}\sum_{E_j \in [1, 1+h]}\left\lvert  \int_{ M_\tau } a^{\zeta_0}_{\delta(h)} (\zeta) \right.  & \left. \lvert U_{j}(\zeta) \rvert^2\,d\mu(\zeta) - \int_0^{\tau_0} \! \int_{\partial M_\tau}  \frac{a^{\zeta_0}_{\delta(h)} (Z) }{\mu_{\tau}(\partial M_\tau)}  \,d\mu_\tau(Z) d\tau \right\rvert^2 \\
&\le  4\delta(h) \int_{ \sqrt{\rho}(\zeta_0) - 2\delta(h)}^{\sqrt{\rho}(\zeta_0) + 2\delta(h)}  C \delta(h)^{2(2n-1)} \lvert\log h\rvert^{-\beta} \,d\tau\\
&= \ocal(  \delta(h)^{4n} \lvert\log h\rvert^{-\beta}) + \ocal(h^2\delta(h)^{-2}).
\end{align}
This completes the proof.
\end{proof}

\subsection{Proof of Theorem~\ref{PROP:VOLUME} using Proposition~\ref{PROP:GRAUERTQEBULK}}

We now have enough tools to tackle the key volume comparison estimate Theorem~\ref{PROP:VOLUME}, which is a Grauert tube analogue of Theorem~\ref{MAINHAN}. The proof uses the covering argument of  \cite[\S3.2]{HezariRiviere16}, \cite[\S5.2]{Han15small}, \cite[\S4.2]{ChangZelditch17}. In what follows we revert to using $\lambda$-notation. Recall from \eqref{EQN:H} that the semi-classical $h$-notation in Proposition~\ref{PROP:GRAUERTQE}--\ref{PROP:GRAUERTQEBULK}; in particular we have $\delta(h) = \lvert \log h \rvert^{-\alpha} = (\log \lambda)^{-\alpha} = \epsilon(\lambda)$.

\begin{proof}[Proof of Theorem~\ref{PROP:VOLUME}]
Let $\tau_0, \tau_1$ be fixed with $0 < \tau_1 < \tau_0$. In what follows we work with centers $\zeta_k$ that lie in the fixed `annulus' $M_{\tau_0} \setminus M_{\tau_1}$, on which the errors remain uniform estimates. As in \cite[Lemma 5.1]{Han15small}, for every $\epsilon(\lambda)$, there exists a {\it log-good} cover  $$\ucal_{\lambda}: = \{\bcal(\zeta_k,\epsilon(\lambda))\}_{k = 1}^{R(\epsilon(\lambda))}$$ of $M_{\tau_0} \setminus M_{\tau_1}$ by balls of radii $ c\epsilon(\lambda)$ such that
\begin{itemize}
\item[(i)] The number $R(\epsilon(\lambda))$ of elements in the covering satisfies $c_1\epsilon(\lambda)^{-2n} \le R(\epsilon(\lambda)) \le c_2 \epsilon(\lambda)^{-2n}$, where $c_1, c_2$ are independent of $\epsilon(\lambda)$.
\item[(ii)] Any $\bcal(\zeta',\epsilon(\lambda)) \subset  M_{\tau_0} \setminus M_{\tau_1}$ is covered by at most $c_3$ (independent of $\epsilon(\lambda)$) number of elements of $\ucal_\lambda$.
\item[(iii)] Any $\bcal(\zeta',\epsilon(\lambda)) \subset M_{\tau_0} \setminus M_{\tau_1}$ contains at least one element of $\{\bcal(\zeta_k, \frac{1}{3} \epsilon(\lambda))\}_{k = 1}^{R(\epsilon(\lambda))}$.
\end{itemize}

We proceed to provide the extraction argument. For each 
\begin{equation}\label{EQN:J}
\lambda_j \in [\lambda, \lambda+1], \quad 1 \le k \le R(\epsilon(\lambda)),
\end{equation}
Set
\begin{equation}
X_{j,k} := \left\lvert  \int_{ M_{\tau_0} } a^{\zeta_k}_{\epsilon(\lambda)} (\zeta) \lvert U_j \rvert^2\,d\mu - \int_0^{\tau_0} \! \int_{\partial M_\tau}  \frac{a^{\zeta_k}_{\epsilon(\lambda_j)} (\zeta) }{\mu_{\tau}(\partial M_\tau)}  \,d\mu_\tau d\tau \right\rvert^2.
\end{equation}
(The two subscripts $j,k$ correspond to the subscript $j$ for the eigenvalue $\lambda_j$ and the subscript $k$ for the points $\zeta_k$.) Also, let $\beta'> 0$ be a parameter to be chosen later and define `exceptional sets' by
\begin{equation}
\Lambda_{k} := \bigg\{ j \colon \lambda_j \in [\lambda, \lambda+1], \; X_{j, k} \ge \epsilon(\lambda)^{4n}(\log \lambda)^{-\beta'}\bigg\}.
\end{equation}
We claim 
\begin{equation}\label{EQN:DENSITYZERO}
\frac{\# \Lambda_{k}}{\lambda^{n-1}} \leq C (\log \lambda)^{-\beta +\beta'}.
\end{equation}
Indeed, this follows from Markov's inequality $\mathbb{P}(X_{j,k} \geq x) \leq x^{-1}{\mathbb{E}} X_{j,k}$. We view $X_{j,k}$ as real-valued random variables index by $j$. The probability measure is the normalized counting measure on the set of indices $j$ satisfying \eqref{EQN:J}. Thanks to Proposition~\ref{PROP:GRAUERTQEBULK}, for all such $j$ the expected value of this random variable is
\begin{equation}
\mathbb{E}X_{j,k} = \ocal(\epsilon(\lambda)^{4n} (\log \lambda)^{-\beta}),
\end{equation}
with the error is uniform in $\zeta_k \in M_{\tau_0} \setminus M_{\tau_1}$ for $k = 1, 2, \dotsc, R(\epsilon(\lambda))$. Finally, setting $x = \epsilon(\lambda)^{4n}(\log \lambda)^{-\beta'}$ in the inequality yields \eqref{EQN:DENSITYZERO}.

Moreover, the union
\begin{equation}
\Lambda := \bigcup_{k = 1}^{R(\epsilon(\lambda))} \Lambda_k
\end{equation}
of the exceptional sets satisfies
\begin{equation}\label{EQN:TOZERO}
\frac{\# \Lambda}{\lambda^{n-1}} \leq C R(\epsilon(\lambda)) (\log \lambda)^{-\beta +\beta'} = C \epsilon(\lambda)^{-2n}(\log \lambda)^{-\beta +\beta'} = C (\log \lambda)^{2n\alpha -\beta + \beta'}.
\end{equation}
Recall from Proposition~\ref{PROP:GRAUERTQEBULK} that $0 < \beta < 1 - 2\alpha(2n-1)$, so $\beta' > 0$ can always be chosen small enough such that the quantity \eqref{EQN:TOZERO} tends to zero whenever $2n\alpha - (1 - 2\alpha(2n-1)) < 0$. This corresponds to the range of $\alpha$ in the statement of Theorem~\ref{PROP:VOLUME}.

Consider now the `generic set'  
\begin{equation} \label{eqn:GENERICSET}
\Sigma : = \{j \colon \lambda_j \in [\lambda, \lambda+1]\} \setminus \Lambda,
\end{equation} 
which is by construction a subsequence of full density:
\begin{equation}
\frac{\# \Sigma}{\lambda^{n-1}} \geq 1 - C \epsilon(\lambda)^{-2n}(\log \lambda)^{-\beta +\beta'} \rightarrow 1.
\end{equation}
If $j \in \Sigma$, then we must have
\begin{equation}
\left\lvert  \int_{ M_{\tau_0} } a^{\zeta_k}_{\epsilon(\lambda_j)} (\zeta) \lvert U_j \rvert^2\,d\mu - \int_0^{\tau_0} \! \int_{\partial M_\tau}  \frac{a^{\zeta_k}_{\epsilon(\lambda_j)} (\zeta) }{\mu_{\tau}(\partial M_\tau)}  \,d\mu_\tau d\tau \right\rvert^2 \le \epsilon(\lambda)^{4n}(\log \lambda)^{-\beta'}
\end{equation}
simultaneously for all $k = 1, 2, \dotsc, R(\epsilon(\lambda))$, that is,
\begin{equation}
\int_{ M_{\tau_0} } a^{\zeta_k}_{\epsilon(\lambda_j)} (\zeta) \lvert U_{j} \rvert^2\,d\mu \le C \Vol_\omega (\bcal(\zeta_k,\epsilon(\lambda_j))) + o(\epsilon(\lambda)^{2n}(\log \lambda)^{-\beta'/2}).
\end{equation}
If $\zeta' \in M_\tau \setminus M$ is an arbitrary point, then the ball $\bcal(\zeta',\epsilon(\lambda_j))$ is contained in at most $c_2$ number (independent of $\lambda$)  of elements of the log-good cover $\ucal_\lambda$, whence we obtain the upper bound 
\begin{equation}
\int_{\bcal(\zeta',\epsilon(\lambda_j))} |U_j|^2\,d\mu  \le C \sum_{\ell=1}^{c_2}  \Vol_\omega (\bcal(\zeta_{k_\ell},\epsilon(\lambda_j))) + o(\epsilon(\lambda)^{2n}(\log \lambda)^{-\beta'/2}) \le C\Vol(\bcal(\zeta', \epsilon(\lambda_j)).
\end{equation}
The constant $C = C(M,g)$ is independent of $\zeta'$ throughout.

It remains to extract another full density subsequence $\Sigma'$ using symbols of the form $b^{\zeta_0}_\epsilon(\zeta) := b(\zeta/\epsilon)$ in local coordinates centered at $\zeta_0$. Here, $0 \le b \le 1$ is taken to be a smooth cut-off function that equals $1$ on $B(0,1/6) \subset \C^n$ and vanishes outside $B(0, 1/3) \subset \C^n$. Repeating the same arguments, we see that for $j \in \Sigma'$, we have
\begin{equation}
\int_{\bcal(\zeta_k, \epsilon(\lambda_j)/3)} |U_j|^2\,d\mu \ge c\Vol(\bcal(\zeta_k, \epsilon(\lambda_j)/6)) - o(\lvert \log \lambda \rvert^{-\beta'/2})
\end{equation}
simultaneously for all $k = 1, 2, \dotsc, R(\epsilon(\lambda))$. Let $\zeta' \in M_\tau \setminus M$ be arbitrary. Every ball $\bcal(\zeta',\epsilon(\lambda_j))$ contains at least one element $\bcal(\zeta',\epsilon(\lambda_j)/3) \in \ucal_\lambda$ of the log-good cover, whence
\begin{equation}
\int_{\bcal(\zeta',\epsilon(\lambda_j))} |U_j|^2\,dV \ge c\Vol(\bcal(\zeta_{k_0},\epsilon(\lambda_j)/3)) \ge  c \Vol(\bcal(\zeta',\epsilon(\lambda_j))).
\end{equation}
Again, it is easy to verify that $c = c(M,g)$ is independent of $\zeta'$. This is the statement of the volume lower bound.

The intersection $\Gamma = \Sigma \cap \Sigma'$ is again a full density subsequence. By construction, every $j \in \Gamma$ satisfies the two-sided bound:
\begin{equation}
c \Vol_\omega(\bcal(\zeta',\epsilon(\lambda_j))) \le \int_{\bcal(\zeta',\epsilon(\lambda_j))} |U_j|^2\,d\mu \le C \Vol(\bcal(\zeta',\epsilon(\lambda_j))) \quad \text{for all $\zeta' \in M_\tau \setminus M$}. 
\end{equation}
This completes the proof of Theorem~\ref{PROP:VOLUME}.
\end{proof}

\section{Log-scale equidistribution of complex zeros: Proof of Theorem~\ref{ZEROSTHintro}}

Recall from the previous section the two key objects of study:
\begin{equation}
\Theta_j(\zeta) := 	\| \phi_j^\C \mid_{\sqrt{\rho}(\zeta)} \|_{L^2(M_{\sqrt{\rho}(\zeta)})} \qquad \text{and} \qquad U_j(\zeta) := \frac{\phi_j^\C(\zeta)}{\Theta_j(\zeta)}.
\end{equation}
By the  Poincar\'{e}-Lelong formula \cite[p.388, Lemma]{GriffithsHarris78},  the current of integration $[\Z_j]$  over the zero set $\Z_j = \{ \zeta \in M_{\tau_0} : \phi_j^\C(\zeta) = 0\}$  is given by the identity
\begin{equation}\label{EQN:PL}
\frac{i}{2\pi} \ddbar \log \lvert U_j \rvert^2 = \frac{i}{2\pi} \ddbar \log \lvert \phi_j^\C \rvert^2 - \frac{i}{2\pi} \ddbar \log \Theta_j^2 = [\Z_j] - \frac{i}{2\pi} \ddbar \log \Theta_j^2.
\end{equation}

To study the currents $[\Z_j]$ at logarithmic length scales, let $D_{\epsilon(\lambda_j)}^{\zeta_0*}$ denote the corresponding pullback operator corresponding to the local holomorphic dilation map \eqref{EQN:DIACPX}. This allows us to work not on shrinking balls $B(\zeta_0, \epsilon(\lambda_j))$ but on a fix-sized ball  $B(\zeta_0,1)$, which is more convenient. The (normalized) small-scale version of \eqref{EQN:PL} becomes
\begin{multline}\label{EQN:PLLOG}
\frac{i}{2\pi \lambda_j\epsilon(\lambda_j)} \ddbar D_{\epsilon(\lambda_j)}^{\zeta_0*}\log \lvert {U}_j \rvert^2 \\
= \frac{1}{\lambda_j\epsilon(\lambda_j)}D_{\epsilon(\lambda_j)}^{\zeta_0*}[{\Z}_j] - \frac{i}{2\pi \lambda_j\epsilon(\lambda_j)} \ddbar D_{\epsilon(\lambda_j)}^{\zeta_0*}\log {\Theta}_j^2\quad \text{as currents on $B(\zeta_0,1)$}. 
\end{multline}
We used the fact that the local dilation map $D_{\epsilon(\lambda_j)}^{\zeta_0}$, being holomorphic, commutes with $\ddbar$.

\begin{rem}
The $\lambda_j^{-1}$ normalization is already present in \eqref{EQN:ZELDITCH}, due to \cite{Zelditch07complex}. Here there is an additional factor of $\epsilon(\lambda_j)^{-1}$, which comes from the proof of Proposition~\ref{PROP:RHS}, specifically \eqref{EQN:METRICTAYLOR}.
\end{rem}

\subsection{Proof of Theorem~\ref{ZEROSTHintro} using Theorem~\ref{PROP:VOLUME}}

%

We rescale the convergence statement \eqref{EQN:ZEROSINTRO} as in  \eqref{EQN:PLLOG}, so that the various objects are defined on a fixed-sized ball $B(\zeta_0, 1)$ that does not change with respect to the frequency $\lambda$. 

 We point out a subtlety  involving the parameter $\alpha > 0$ in the proof of Theorem~\ref{ZEROSTHintro} using Theorem~\ref{PROP:VOLUME}. Namely, if a full density subsequence satisfies volume comparison \eqref{EQN:VOLCOMPintro} at length scale $\epsilon(\lambda_j) = (\log \lambda_j)^{-\alpha}$, then it satisfies the zeros distribution result \eqref{EQN:ZEROSINTRO} at a coarser length scale $\epsilon'(\lambda_j) := (\log \lambda_j)^{-\alpha'}$ for any $\alpha ' < \alpha$. This inequality is strict -- see the argument around \eqref{EQN:WDELTA}--\eqref{EQN:RESCALEVOL}. To emphasize the role of the two scales,  we restate Theorem~\ref{ZEROSTHintro}--\ref{PROP:VOLUME} as follows.

\begin{theo}\label{ZEROSTH2}
Let $(M,g)$ be a real analytic, negatively curved, compact manifold without boundary.  Let $\omega := -i \ddbar \rho$ denote the \kahler form on the Grauert tube $M_{\tau_0}$. Assume that
\begin{equation}
0 \le \alpha' < \frac{1}{2(3n-1)}, \quad \epsilon'(\lambda_j) = (\log \lambda_j)^{-\alpha'}.
\end{equation}
Then there exists a full density subsequence of eigenvalues $\lambda_{j_k}$ such that for arbitrary but fixed $\zeta_0 \in M_{\tau_0} \backslash M$, there is a uniform two-sided volume bound
\begin{equation}\label{EQN:VOLCOMP}
c\Vol_\omega(\bcal(\zeta_0,\epsilon'(\lambda_{j_k}))) \le \int_{\bcal(\zeta_0,\epsilon'(\lambda_{j_k}))} \lvert U_{j_k}\rvert^2 d\mu \le C\Vol_\omega(\bcal(\zeta_0,\epsilon'(\lambda_{j_k}))).
\end{equation}
The constants $c, C$ are geometric constants depending only on $\sqrt{\rho}(\zeta_0)$; they are uniform for $\zeta_0$ lying in an `annulus' $0 < \tau_1 \le \sqrt{\rho}(\zeta_0) \leq \tau_0$.

Moreover, for any $\alpha$ satisfying
\begin{equation}
0 \le \alpha < \alpha' < \frac{1}{2(3n-1)}, \quad \epsilon(\lambda_j) = (\log \lambda_j)^{-\alpha},
\end{equation}
the full density subsequence satisfying \eqref{EQN:VOLCOMP} also satisfies
\begin{equation}\label{EQN:ZEROS}
\frac{1}{\lambda_{j_k}\epsilon(\lambda_{j_k})} D_{\epsilon(\lambda_j)}^{\zeta_0*}[\Z_{\lambda_{j_k}}] \rightharpoonup \frac{i}{\pi}\ddbar\lvert \Im(\zeta - \zeta_0) \rvert_{g_0}
\quad \text{as currents on  $B(\zeta_0,1)$}.
\end{equation}
Here,  $D_{\epsilon(\lambda_j)}^{\zeta_0*}$ denote pullback by the local holomorphic dilation \eqref{EQN:DIACPX} and ${g_0}$ denotes the flat metric. Equivalently, for every test form $\eta \in \dcal^{(n-1,n-1)}(B(\zeta_0,1))$,
\begin{equation}
\int_{B(\zeta_0,1)}\eta \wedge\frac{1}{\lambda_{j_k}\epsilon(\lambda_{j_k})} D_{\epsilon(\lambda_j)}^{\zeta_0*} [{\Z}_{\lambda_{j_k}}] = \int_{B(\zeta_0,1)} \eta \wedge \frac{i}{\pi} \ddbar \lvert \Im(\zeta - \zeta_0) \rvert_{g_0} + o(1).
\end{equation}
\end{theo}

\begin{rem}
By a partition of unity argument, Theorem~\ref{ZEROSTH2} for general test forms supported on \kahler balls implies Theorem~\ref{ZEROSTHintro} for test forms on $M_{\tau_0}$ of the form $f \omega^{n-1}$ with $f \in C(M_{\tau_0})$.
\end{rem}

The volume comparison \eqref{EQN:VOLCOMP} has already been proved in the previous section. Comparing what is left to prove -- namely \eqref{EQN:ZEROS} -- with the identity \eqref{EQN:PLLOG}, we see that it suffices to establish the following Propositions~\ref{PROP:RHS}--\ref{PROP:LHS}.



%
%

\begin{prop}\label{PROP:RHS}
For the entire sequence of eigenvalues $\lambda_j$, for every $\zeta_0 \in M_{\tau_0} \backslash M$, we have
\begin{equation}
\frac{i}{2\pi\lambda_j\epsilon(\lambda_j)} \ddbar D_{\epsilon(\lambda_j)}^{\zeta_0*}\log {\Theta}_j^2 \rightarrow \frac{i}{\pi} \ddbar \lvert \Im(\zeta - \zeta_0) \rvert_{g_0} \quad \text{as currents on $B(\zeta_0,1)$}.
\end{equation}
Here, $\lvert \, \cdot \, \rvert_{g_0}$ denotes the Euclidean distance.
\end{prop}

\begin{prop}\label{PROP:LHS}
There exists a full density subsequence of eigenvalues $\lambda_{j_k}$ such that, for every $\zeta_0 \in M_{\tau_0} \backslash M$, we have
\begin{itemize}
\item[(i)] $(\lambda_{j_k}\epsilon(\lambda_{j_k}))^{-1}\log D_{\epsilon(\lambda_{j_k})}^{\zeta_0*}\lvert {U}_{j_k} \rvert^2 \rightarrow 0$ strongly in $L^1(B(\zeta_0,1))$;
\item[(ii)] $(\lambda_{j_k}\epsilon(\lambda_{j_k}))^{-1}\ddbar \log D_{\epsilon(\lambda_{j_k})}^{\zeta_0*}\lvert {U}_{j_k} \rvert^2 \rightharpoonup 0$ weakly in $\dcal^{(n-1,n-1)'}(B(\zeta_0,1))$.
\end{itemize}
\end{prop}

\subsection{Proof of Proposition~\ref{PROP:RHS} using pseudodifferential operators}

Using \eqref{EQN:CPXEF}, we see
\begin{equation}\label{EQN:IDENTITY}
\phi_j^\C(\zeta) = e^{\lambda_j\sqrt{\rho}(\zeta)} (P^{\sqrt{\rho}(\zeta)} \phi_j)(\zeta), \qquad \zeta \in M_{\tau_0}.
\end{equation}
Therefore,
\begin{align}
D_{\epsilon(\lambda_j)}^{\zeta_0*}{\Theta}_j(\zeta)^2 &= D_{\epsilon(\lambda_j)}^{\zeta_0*}\big\| \phi_j^\C \mid_{\partial M_{\sqrt{\rho}(\zeta)}} \big\|_{L^2(\partial M_{\sqrt{\rho}(\zeta)})}^2\\
& = \bigg\| \phi_j^\C \mid_{\partial M_{D_{\epsilon(\lambda_j)}^{\zeta_0*}\sqrt{\rho}(\zeta)}} \bigg\|_{L^2\big(\partial M_{D_{\epsilon(\lambda_j)}^{\zeta_0*}\sqrt{\rho}(\zeta)}\big)}^2\\
&= \left\langle \Pi_{D_{\epsilon(\lambda_j)}^{\zeta_0*}\sqrt{\rho}(\zeta)} \phi_j^\C, \Pi_{D_{\epsilon(\lambda_j)}^{\zeta_0*}\sqrt{\rho}(\zeta)} \phi_j^\C \right\rangle_{L^2\big(\partial M_{D_{\epsilon(\lambda_j)}^{\zeta_0*}\sqrt{\rho}(\zeta)}\big)} \\
&= e^{2\lambda_j D_{\epsilon(\lambda_j)}^{\zeta_0*}\sqrt{\rho}(\zeta)} \left\langle P^{D_{\epsilon(\lambda_j)}^{\zeta_0*}\sqrt{\rho}(\zeta) *} \Pi_{D_{\epsilon(\lambda_j)}^{\zeta_0*}\sqrt{\rho}(\zeta)} P^{D_{\epsilon(\lambda_j)}^{\zeta_0*}\sqrt{\rho}(\zeta)}  \phi_j, \phi_j \right\rangle_{L^2(M)}. \label{EQN:THETATILDE}
\end{align}
The last equality follows from \eqref{EQN:IDENTITY}.

The operators
\begin{equation}\label{EQN:OPERATORA}
A(\epsilon(\lambda_j), \sqrt{\rho}(\zeta)) := P^{D_{\epsilon(\lambda_j)}^{\zeta_0*}\sqrt{\rho}(\zeta) *} \Pi_{D_{\epsilon(\lambda_j)}^{\zeta_0*}\sqrt{\rho}(\zeta)} P^{D_{\epsilon(\lambda_j)}^{\zeta_0*}\sqrt{\rho}(\zeta)} \in \Psi^{-\frac{n-1}{2}}(M)
\end{equation}
forms  an analytic family in the parameter $\sqrt{\rho}(\zeta) \in (0,\tau_0]$ with $A(\epsilon(\lambda_j), \sqrt{\rho}(\zeta)) \rightarrow \operatorname{Id}$ as $\sqrt{\rho}(\zeta) \rightarrow 0$. It is easy to see using the Schur-Young test that $(1 + \Delta)^{-\frac{n+1}{2}}A(\epsilon) \in \Psi^{-n}(M)$ is a uniformly upper bounded family of operators on $L^2(M)$ (see \cite[(34)]{Zelditch07complex}). Therefore, writing $A(\epsilon(\lambda_j), \sqrt{\rho}(\zeta)) = (1 + \lambda_j)^{\frac{n+1}{2}}(1 + \Delta)^{-\frac{n+1}{2}}A(\epsilon(\lambda_j), \sqrt{\rho}(\zeta))$, we find
\begin{equation}\label{EQN:NORMBOUND}
\left\lvert\frac{1}{\lambda_j} \log \left\langle P^{D_{\epsilon(\lambda_j)}^{\zeta_0*}\sqrt{\rho}(\zeta) *} \Pi_{D_{\epsilon(\lambda_j)}^{\zeta_0*}\sqrt{\rho}(\zeta)} P^{D_{\epsilon(\lambda_j)}^{\zeta_0*}\sqrt{\rho}(\zeta)}  \phi_j, \phi_j \right\rangle_{L^2(M)} \right\rvert\le C \frac{\log \lambda_j}{\lambda_j}
\end{equation}
for some $C$ independent of $\epsilon$. Combining \eqref{EQN:THETATILDE} and \eqref{EQN:NORMBOUND} gives
\begin{equation}\label{EQN:CONVERGE}
\frac{1}{2\pi  \lambda_j\epsilon(\lambda_j)} \log D_{\epsilon(\lambda_j)}^{\zeta_0*}{\Theta}_j(\zeta)^2 = \frac{1}{\pi \epsilon(\lambda_j)}D_{\epsilon(\lambda_j)}^{\zeta_0*}\sqrt{\rho}(\zeta) + \ocal(\lambda_j^{-1}\log \lambda_j).
\end{equation}
Recall from Section~\ref{SEC:BACKGROUND} that the Grauert tube function $\rho$ is related to the complexified Riemannian distance function $r$ on $M_\C \times \overline{M}_\C$ by
\begin{equation}
\rho(\zeta) = -\frac{1}{4} r^2(\zeta, \bar{\zeta}), \qquad \zeta = \exp_x^\C(i\xi) \in M_{\tau_0}.
\end{equation}
Taylor expanding the metric yields $\sqrt{\rho}(\zeta) = \lvert \Im(\zeta - \zeta_0) \rvert_{g_0} + O(\lvert \Im(\zeta - \zeta_0) \rvert^2_{g_0})$, in which $\lvert \,\cdot\, \rvert_{g_0}$ denotes the flat metric. This gives rise to the $\lambda_j \to \infty$ asymptotics
\begin{equation}\label{EQN:METRICTAYLOR}
D_{\epsilon(\lambda_j)}^{\zeta_0*}\sqrt{\rho}(\zeta) = \epsilon(\lambda_j) \lvert \Im(\zeta - \zeta_0) \rvert_{g_0} + O(\epsilon(\lambda_j)^2), \qquad \zeta = \exp^\C_x(i\xi) \in M_{\tau_0}.
\end{equation}
The statement of Proposition~\ref{PROP:RHS} is now an immediate consequence of \eqref{EQN:CONVERGE} and \eqref{EQN:METRICTAYLOR}.

\subsection{Proof of Proposition~\ref{PROP:LHS} using subharmonic function theory}
Proposition~\ref{PROP:LHS} is modeled after arguments that  have appeared in \cite{ShiffmanZelditch99, Zelditch07complex, ChangZelditch17}.
Given $\zeta_0 \in M_{\tau_0} \backslash M$, consider the family of plurisubharmonic functions
\begin{equation}
v_j := \frac{1}{\lambda_j\epsilon(\lambda_j)} \log D_{\epsilon(\lambda_j)}^{\zeta_0*}\lvert \phi_j^\C \rvert^2 \in \operatorname{PSH}(B(\zeta_0,1)).
\end{equation}
(The functions $v_j$ are indeed subharmonic because $\phi_j^\C$ are holomorphic by constrcution.) We claim
\begin{itemize}
\item[(i)] $\{v_j\}$ is uniformly bounded above on $B(\zeta_0,1)$;

\item[(ii)] $\limsup_{j \to \infty} v_j(\zeta) \le 2 \sqrt{\rho}(\zeta)$ on $B(\zeta_0,1)$.
\end{itemize}
Notice $\sup_{B(\zeta_0,1)} D_{\epsilon(\lambda_j)}^{\zeta_0*}\lvert {U}_j \rvert^2 = \sup_{\bcal(\zeta_0,\epsilon(\lambda))}\lvert U_j \rvert^2$. To prove the first statement, it suffices to obtain a uniform upper bound on each slice $\partial M_\tau \cap \overline{\bcal(\zeta_0,\epsilon(\lambda_j))}$ that is independent of $\tau$. Since $ u_j^{\tau} \in \ocal^{\frac{n-1}{4}}(\partial M_\tau)$, we see (cf.\ \cite[\S5.1]{Zelditch07complex})
\begin{equation}
\sup_{\partial M_\tau \cap \overline{B(\zeta_0,\epsilon(\lambda_j))}} \lvert U_j \rvert^2 \le \sup_{\partial M_\tau} \lvert u_j^\tau \rvert^2 \le \lambda_j^n \|u_j^\tau\|_{L^2(\partial M_\tau)} = \lambda_j^n.
\end{equation}
Rewriting the left-hand side as $U_j = \phi_j^\C / \| \phi_j^\C \|_{L^2(\partial M_{\sqrt{\rho}})}$, taking the logarithm, dividing by $\lambda_j$, and finally using the limit formula of Proposition~\ref{PROP:RHS} finishes the proof of (i) and (ii).

It follows from a standard compactness theorem on plurisubharmonic functions \cite[Theorem~4.1.9]{Hormandervol1} that either $v_j \rightarrow -\infty$ locally uniformly, or there exists a subsequence that is convergent in $L^1_\mathrm{loc}(B(\zeta_0,1))$. The first possibility is easily ruled out. Indeed, if it were true, then
\begin{multline}
\frac{1}{\lambda_j\epsilon(\lambda_j)} \log D_{\epsilon(\lambda_j)}^{\zeta_0*}\lvert U_j \rvert^2 \le -1 \quad \text{on $B(\zeta_0,1)$ for all $\lambda_j \gg 1$} \\
\iff \lvert U_j \rvert^2 \le e^{-\lambda_j\epsilon(\lambda_j)} \quad \text{on $\bcal(\zeta_0,\epsilon(\lambda_j))$ for all $\lambda_j \gg 1$},
\end{multline}
contradicting the mass comparison assumption \eqref{EQN:VOLCOMP}.

\begin{rem}
By a covering argument similar to the proof of Theorem~\ref{PROP:VOLUME}, it is easy to see that if a sequence $\{U_j\}$ satisfies volume comparison \eqref{EQN:VOLCOMP}, then it satisfies volume comparison at all coarser length scales $\epsilon(\lambda_j) = (\log \lambda_j)^{-\alpha}$ for $\alpha' < \alpha < \frac{1}{2(3n-1)}$.
\end{rem}

Therefore, $v_{j}$ has a subsequence, which we continue to denote by $v_{j}$, that converges in $L^1$ to $v \in L^1(B(\zeta_0,1))$. By passing to yet another subsequence if necessary, we may assume that the convergence to $v$ is pointwise almost everywhere. The upper-semicontinuous regularization
\begin{equation}
v^*(\zeta) := \limsup_{\eta \rightarrow \zeta} v(\eta) \le 2 \sqrt{\rho}(\zeta)
\end{equation}
of $v$ is then a plurisubharmonic function on $B(\zeta_0,1)$ and $v_{j} \to v^*$ pointwise almost everywhere.\footnote{A similar argument is used in  \cite[Lemma 1.4]{ShiffmanZelditch99}, which gives
 further details. See also \cite{Kl} for background.} The upper bound of $2\sqrt{\rho}(\zeta)$ follows from claim (ii) above.

Set
\begin{equation}
\psi := v^* - 2 \sqrt{\rho} \le 0 \quad \text{on $B(\zeta_0,1)$}.
\end{equation}
Assume for purposes of a contradiction that $\| \lambda_j^{-1} \epsilon(\lambda_j)^{-1} \log D_{\epsilon(\lambda_j)}^{\zeta_0*}\lvert {U}_j \rvert^2 \|_{L^1(B(\zeta_0,1)} \ge \delta > 0$. It follows that
\begin{equation}\label{EQN:WDELTA}
W_\delta := \{ \zeta \in B(\zeta_0,1) : \psi(\zeta) < - \delta /2\}
\end{equation}
is an open set  with nonempty interior. The shape of $W_{\delta}$ is unknown -- it may have a very small inradius -- but it is a fixed (independent of $\lambda_j$) open set. To gain control over this unknown set $W_\delta$, we make use of the volume comparison assumption \eqref{EQN:VOLCOMP} that takes place at the finer scale $\epsilon'(\lambda_j) = (\log\lambda_j)^{-\alpha'}$ for $\alpha' < \alpha$. From this assumption we know
\begin{equation}
\int_{B(\zeta',\epsilon'(\lambda_j))} | U_j |^2 \omega^n \ge c\Vol_\omega(\bcal(\zeta_0,\epsilon'(\lambda_j))) \quad \text{for all $\zeta' \in M_{\tau_0} \setminus M$.}
\end{equation}
Rescaling yields
\begin{equation}\label{EQN:RESCALEVOL}
\int_{B(\zeta',\epsilon'(\lambda_j)\epsilon^{-1}(\lambda_j))}D_{\epsilon(\lambda_j)}^{\zeta_0}| U_j |^2 \omega^n \ge c\Vol_\omega(\bcal(\zeta_0,\epsilon'(\lambda_j)\epsilon^{-1}(\lambda_j))).
\end{equation}
Notice in the above integral the radii $\epsilon'(\lambda_j)\epsilon^{-1}(\lambda_j) = \log(\lambda_j)^{-(\alpha' - \alpha)}$ of the domain of integration shrinks to $0$. Therefore, there exists $\zeta' \in M_{\tau_0} \setminus M$ for which $B(\zeta', \epsilon'(\lambda_j)\epsilon^{-1}(\lambda_j)) \subset W_\delta$ for all $\lambda_j$ sufficiently large.

On one hand, from the definition \eqref{EQN:WDELTA}, we know that on all of $W_\delta$ -- and in particular on $B(\zeta', \epsilon'(\lambda_j)\epsilon^{-1}(\lambda_j))$ -- we have the upper bound $\lambda_j^{-1}\epsilon(\lambda_j)^{-1} \log D_{\epsilon(\lambda_j)}^{\zeta'*}\lvert {U}_j \rvert^2 < -\delta/2$, i.e.,
\begin{equation}\label{CONTRA}
D_{\epsilon(\lambda_j)}^{\zeta_0*}  \lvert {U}_j(\zeta) \rvert^2 \le e^{ -\delta\lambda_j \epsilon(\lambda_j) }, \quad\zeta \in B(\zeta', \epsilon'(\lambda_j) \epsilon^{-1}(\lambda_j)), \quad \lambda_j \gg 1.
\end{equation}
Clearly, the exponential decay upper bound \eqref{CONTRA} is incompatible with the logarithmic lower bound \eqref{EQN:RESCALEVOL} as $\lambda_j \to \infty$. This shows by way of contradiction that the original assumption 
$$\| \lambda_j^{-1} \epsilon(\lambda_j)^{-1} \log D_{\epsilon(\lambda_j)}^{\zeta_0*}\lvert {U}_j \rvert^2 \|_{L^1(B(\zeta_0,1)} \ge \delta > 0$$
does not hold, thereby proving Proposition~\ref{PROP:LHS} (i), from which Proposition~\ref{PROP:LHS} (ii) is an immediate consequence. Combining \eqref{EQN:PLLOG}, Proposition~\ref{PROP:RHS}, and Proposition~\ref{PROP:LHS} (ii), we obtain the zeros distribution statement of Theorem~\ref{ZEROSTH2}:
\begin{equation}
\frac{1}{\lambda_{j_k}\epsilon(\lambda_{j_k})} D_{\epsilon(\lambda_j)}^{\zeta_0*}[{Z}_{\lambda_{j_k}}] \rightharpoonup  \frac{i}{\pi}\ddbar\lvert \Im(\zeta - \zeta_0)\rvert_{g_0} \quad \text{as currents on $B(\zeta_0,1)$}
\end{equation}
for a full density subsequence satisfying volume comparison at the finer scale $\alpha'$. This concludes the proof of Theorem~\ref{ZEROSTHintro}.

\appendix

\section{Currents of integration over singular varieties} \label{Appendix:Current}

In general, the zero set $X$ of a holomorphic function on a complex manifold $V$  is called a complex
analytic variety (which could also be the common zeros of finitely many
holomorphic functions).  See for instance \cite{W72}. It has a decomposition into a regular set $R(X)$
and a lower-dimensional singular set $S(X)$, i.e.,  $X = R(X) \cup S(X)$ where $R(X)$ is
a manifold and $\dim S(X) < \dim X$ (see  \cite[Theorem 2.1.8]{K71}).
In \cite[Theorem 3.1.1]{K71} it is proved that if $X$
a $k$-dimensional complex subvariety of a complex manifold $V$ and
$u \in A_c^{2k}(V)$ is a smooth $(2k)$-form then 
\begin{equation} [X](u): =\int_X u = \int_{R(X)} \iota^* u \end{equation}
is a closed current  (due to Lelong
\cite{L57}). King used Federer's geometric measure theory \cite{Fed69} to study such currents. A modern exposition can be found in \cite[Example 1.16]{Dem}.

\subsection{Shiffman's Appendix}

We asked B. Shiffman for further references on currents of integration over
singular analytic varieties. He wrote the following addition to the Appendix,
and refers to   \cite[Lemma A.2]{Sh} for an elementary proof.

Here is a simpler way to 
show that $[X] = [Z_f]$ is a well-defined current: It suffices to show that the set $R(X)$ of smooth points has finite volume in a neighborhood $U$ of a singular point $z_0$. By the Weierstrass preparation theorem applied to $f$, it follows that projections from $X \cap U$ to coordinate hyperplanes have finite fibers of bounded cardinality (for good coordinates) and therefore $\Vol(R(X) \cap U)= \int_{R(X)\cap U}\omega^{n-1} <\infty$.  

The fact that Poincare-Lelong holds at the singular points follows from the fact that the singular set $S(X) $ has Hausdorff $(2n-3)$-dimensional measure $0$,  and therefore $\|\ddbar \log |f| \|(S)=0$,  since the total variation measure of a current of order zero and dimension $p$ vanishes on sets of Hausdorff $p$-measure zero. (In fact, $S(X)$ is a subvariety of real
codimension $4$).

\end{document}